\documentclass[a4paper,11pt,fleqn]{article}
\usepackage{pstricks}
\usepackage{amsmath}
\usepackage{amssymb}
\usepackage{mathrsfs} 
\usepackage{theorem} 
\usepackage{euscript}
\usepackage{exscale,relsize}
\usepackage{graphicx}
\usepackage{srcltx}
\usepackage{xcolor}
\usepackage{charter}
\usepackage{url}
\newcommand{\email}[1]{\href{mailto:#1}{\nolinkurl{#1}}}
\PassOptionsToPackage{normalem}{ulem}
\usepackage{ulem}

\renewcommand{\leq}{\ensuremath{\leqslant}}
\renewcommand{\geq}{\ensuremath{\geqslant}}
\newcommand{\minimize}[2]{\ensuremath{\underset{\substack{{#1}}}%
{\text{\rm minimize}}\;\;#2 }}

\newcommand{\pair}[2]{\langle{{#1},{#2}}\rangle} 
\newcommand{\Pair}[2]{\big\langle{{#1},{#2}}\big\rangle}

\newcommand{\menge}[2]{\big\{{#1}~\big |~{#2}\big\}} 
\newcommand{\Menge}[2]{\Big\{{#1}~\Big |~{#2}\Big\}} 

\newcommand{\IDD}{\ensuremath{\text{\rm int\:dom}f}}
\newcommand{\IDDB}{\ensuremath{\text{\rm int\:dom}\boldsymbol{f}}}

\newcommand{\RX}{\ensuremath{\left]-\infty,+\infty\right]}}

\newcommand{\XX}{\ensuremath{{\mathcal X}}}
\newcommand{\WC}{\ensuremath{{\mathfrak W}}}
\newcommand{\XXX}{\ensuremath{\boldsymbol{\mathcal X}}}

\newcommand{\YY}{\ensuremath{{\mathcal Y}}}
\newcommand{\Sum}{\ensuremath{\displaystyle\sum}}

\newcommand{\emp}{\ensuremath{{\varnothing}}}
\newcommand{\prox}{\ensuremath{\text{\rm prox}}}

\newcommand{\Id}{\ensuremath{\operatorname{Id}}\,}

\newcommand{\RR}{\ensuremath{\mathbb{R}}}

\newcommand{\RP}{\ensuremath{\left[0,+\infty\right[}}

\newcommand{\RPP}{\ensuremath{\left]0,+\infty\right[}}
\newcommand{\NN}{\ensuremath{\mathbb N}}

\newcommand{\weakly}{\ensuremath{\:\rightharpoonup\:}}
\newcommand{\exi}{\ensuremath{\exists\,}}
\newcommand{\pinf}{\ensuremath{{+\infty}}}

\newcommand{\dom}{\ensuremath{\text{\rm dom}\,}}

\newcommand{\Cart}{\ensuremath{\raisebox{-0.5mm}{\mbox{\huge{$\times$}}}}}

\newcommand{\sign}{\ensuremath{\text{\rm sign}}}
\newcommand{\inte}{\ensuremath{\text{\rm int}\,}}
\newcommand{\intdom}{\ensuremath{\text{\rm int\,dom}\,}}

\newcommand{\ran}{\ensuremath{\text{\rm ran}\,}}
\newcommand{\zer}{\ensuremath{\text{\rm zer}\,}}
\newcommand{\gra}{\ensuremath{\text{\rm gra}\,}}

\newcommand{\Fix}{\ensuremath{\text{\rm Fix}\,}}

\newtheorem{theorem}{Theorem}[section]
\newtheorem{lemma}[theorem]{Lemma}
\newtheorem{corollary}[theorem]{Corollary}
\newtheorem{proposition}[theorem]{Proposition}

\theoremstyle{plain}{\theorembodyfont{\rmfamily}%
}
\theoremstyle{plain}{\theorembodyfont{\rmfamily}%
\newtheorem{example}[theorem]{Example}}
\theoremstyle{plain}{\theorembodyfont{\rmfamily}%
\newtheorem{remark}[theorem]{Remark}}
\theoremstyle{plain}{\theorembodyfont{\rmfamily}%
}
\theoremstyle{plain}{\theorembodyfont{\rmfamily}%
\newtheorem{condition}[theorem]{Condition}}
\theoremstyle{plain}{\theorembodyfont{\rmfamily}%
\newtheorem{definition}[theorem]{Definition}}
\theoremstyle{plain}{\theorembodyfont{\rmfamily}
\newtheorem{problem}[theorem]{Problem}}
\theoremstyle{plain}{\theorembodyfont{\rmfamily}
}

\numberwithin{equation}{section}
\definecolor{labelkey}{rgb}{0,0.08,0.45}
\definecolor{refkey}{rgb}{0,0.6,0.0}
\definecolor{Brown}{rgb}{0.45,0.0,0.05}
\definecolor{dgreen}{rgb}{0.00,0.49,0.00}
\definecolor{dblue}{rgb}{0,0.08,0.75}
\RequirePackage[dvips,colorlinks,hyperindex]{hyperref}
\hypersetup{linktocpage=true,citecolor=dblue,linkcolor=dgreen}

\begin{document}

\title{\sffamily\LARGE Solving Composite Monotone Inclusions in
Reflexive Banach Spaces by Constructing Best Bregman Approximations
from Their Kuhn-Tucker Set}

\author{Patrick L. Combettes and Quang Van Nguyen\\[5mm]
\small
\small Sorbonne Universit\'es -- UPMC Univ. Paris 06\\
\small UMR 7598, Laboratoire Jacques-Louis Lions\\
\small F-75005 Paris, France\\
\small \email{plc@ljll.math.upmc.fr},\: 
\email{quangnv@ljll.math.upmc.fr}
}
\date{~}

\maketitle
\setcounter{page}{1}

\vskip 8mm

\centerline{\large\em In memory of Jean Jacques Moreau
(1923--2014)}

\begin{abstract}
We introduce the first operator splitting method for composite 
monotone inclusions outside of Hilbert spaces. The proposed 
primal-dual method constructs iteratively the best Bregman 
approximation to an arbitrary point from the Kuhn-Tucker set 
of a composite monotone inclusion. 
Strong convergence is established in reflexive
Banach spaces without requiring additional restrictions on the
monotone operators or knowledge of the norms of the linear
operators involved in the model. The monotone operators are
activated via Bregman distance-based resolvent operators.
The method is novel even in Euclidean spaces, where it provides 
an alternative to the usual proximal methods based on the 
standard distance. 
\end{abstract}

{\bfseries Key words.}
Best approximation,
Banach space,
Bregman distance,
duality,
Legendre function,
monotone operator,
operator splitting,
primal-dual algorithm.

\newpage
\section{Introduction}
\label{IVsec:I}

Let $\XX$ be a reflexive real Banach space with norm $\|\cdot\|$
and let $\pair{\cdot}{\cdot}$ be the duality pairing between
$\XX$ and its topological dual $\XX^*$. A set-valued operator 
$M\colon\XX\to 2^{\XX^*}$ with graph
$\gra M=\menge{(x,x^*)\in\XX\times\XX^*}{x^*\in Mx}$ is monotone if 
\begin{equation}
(\forall (x_1,x_1^*)\in\gra M)(\forall (x_2,x_2^*)\in\gra M)\quad
\pair{x_1-x_2}{x_1^*-x_2^*}\geq 0,
\end{equation}
and maximally monotone if, furthermore, there exists no monotone 
operator from $\XX$ to $2^{\XX^*}$ the graph of which properly 
contains $\gra M$. Monotone operator theory emerged in the early 
1960s as a well-structured branch of nonlinear analysis 
\cite{Kach60,Mint61,Mint62,Zara60}, and its remains very active 
\cite{Livre1,Borw10,Simo08,ZeidXX}.
One of the main reasons for the success of the theory is that a
significant range of problems in areas such as optimization, 
economics, variational inequalities, partial differential 
equations, mechanics, signal and image processing, optimal 
transportation, machine learning, and traffic theory can be 
reduced to solving inclusions of the type 
\begin{equation}
\label{e:hanoi-oct2011}
\text{find}\;\;x\in\XX\quad\text{such that}\quad 0\in Mx,
\end{equation}
where $M\colon\XX\to 2^{\XX^*}$ is maximally monotone.
Conceptually, this inclusion can be solved via the Bregman 
proximal point algorithm, special instances of which go back to 
\cite{Ecks93,Kass85,Tebo92}. 
To present its general form \cite{Sico03}, 
we need the following definitions, which revolve around the 
notion of a Bregman distance pioneered in \cite{Breg67}.

\begin{definition}{\rm\cite{Ccm01,Sico03}}
Let $\XX$ be a reflexive real Banach space and let 
$f\colon\XX\to\RX$ be a proper lower semicontinuous convex function,
with conjugate $f^*\colon\XX^*\to\RX\colon
x^*\mapsto\sup_{x\in\XX}(\pair{x}{x^*}-f(x))$ and 
Moreau subdifferential \cite{Mor63c}
\begin{equation}
\label{e:subdiff}
\partial f\colon\XX\to 2^{\XX^*}\colon x\mapsto
\menge{x^*\in\XX^*}{(\forall y\in\XX)\,
\pair{y-x}{x^*}+f(x)\leq f(y)}.  
\end{equation}
Then $f$ is a \emph{Legendre function} if it is 
\emph{essentially smooth} in the sense that $\partial f$ is 
both locally bounded and single-valued on its
domain, and \emph{essentially strictly convex} in the sense
that $\partial f^*$ is locally bounded on its domain and 
$f$ is strictly convex on every convex subset of $\dom\partial f$.
Moreover, $f$ is G\^ateaux differentiable on $\IDD\neq\emp$
and the associated \emph{Bregman distance} is 
\begin{equation}
\label{e:Bdist}
\begin{aligned}
D^f\colon\XX\times\XX&\to\,[0,\pinf]\\
(x,y)&\mapsto 
\begin{cases}
f(x)-f(y)-\pair{x-y}{\nabla f(y)},&\text{if}\;\;y\in\IDD;\\
\pinf,&\text{otherwise}.
\end{cases}
\end{aligned}
\end{equation}
Let $C$ be a closed convex subset of $\XX$ such that
$C\cap\IDD\neq\emp$. The \emph{Bregman projector} onto $C$ induced
by $f$ is 
\begin{equation}
\label{e:2001}
\begin{aligned}
P^f_C\colon\IDD&\to C\cap\IDD\\
y&\mapsto\underset{x\in C}{\text{argmin}}\,D^f(x,y).
\end{aligned}
\end{equation}
\end{definition}

The fact that, for every $y\in\IDD$, $P^f_Cy\in\IDD$ exists and is 
unique is established in \cite[Corollary~7.9]{Ccm01}.
It follows from \cite[Theorem~5.18]{Sico03} that, under suitable 
assumptions on $f$ and $M$, given a sequence $(\gamma_n)_{n\in\NN}$ 
in $\RPP$ such that $\inf_{n\in\NN}\gamma_n>0$, the sequence 
defined by
\begin{equation}
\label{epOlk28ha10}
x_0\in\IDD\quad\text{and}\quad(\forall n\in\NN)\quad
x_{n+1}=(\nabla f+\gamma_n M)^{-1}\circ\nabla f(x_n)
\end{equation}
converges weakly to a solution to \eqref{e:hanoi-oct2011} (in the
case when $\XX$ is a Hilbert space and $f=\|\cdot\|^2/2$, 
$(\nabla f+\gamma_n M)^{-1}\circ\nabla f$ reduces to the standard 
resolvent $J_{\gamma_n M}$ and we obtain
the classical result of \cite[Theorem~1]{Rock76}). A strongly
convergent variant of \eqref{epOlk28ha10} was proposed in 
\cite{Pams03}. In applications,
however, $M$ is typically too complex for \eqref{epOlk28ha10} to be
implementable. For instance, given a real Banach space $\YY$,
a typical composite model is 
$M=A+L^*BL$, where $A\colon\XX\to 2^{\XX^*}$ and 
$B\colon\YY\to 2^{\YY^*}$ are monotone, and $L\colon\XX\to\YY$ 
is linear and bounded. 
In Hilbert spaces, if $\XX=\YY$ and $L=\Id$, several 
well-established splitting methods are available to solve 
\eqref{e:hanoi-oct2011}, i.e., to find a zero of $A+B$ 
using $A$ and $B$ separately at each iteration
\cite{Livre1,Lion79,Merc79,Tsen00}. Splitting methods for the 
more versatile composite model $M=A+L^*BL$ in Hilbert spaces 
were first proposed in \cite{Siop11} (see 
\cite{Optl14,Bot13a,Bot13c,Siop13,Opti14,Bang13} for subsequent 
developments). These methods provide in general only weak 
convergence to an unspecified solution and, in addition, they 
require knowledge of
$\|L\|$ or potentially costly inversions of linear operators. 
The recent method primal-dual method of \cite{Nfao15} circumvents
these limitations and, in addition, converges to the best
approximation to a reference point from the Kuhn-Tucker set
relative to the underlying hilbertian distance. 
The objective of this paper is to extend it to
reflexive Banach spaces and to best approximation relative
to general Bregman distances. Let us stress that the theory of 
splitting algorithms in Banach spaces is rather scarce as most 
hilbertian splitting methods cannot be naturally extended to that
setting; in particular, to the best of our knowledge there exists 
at present no splitting algorithm for finding a zero of $M=A+L^*BL$ 
outside of Hilbert spaces. By contrast, the geometric primal-dual
construction of \cite{Nfao15}, which consists in projecting a 
reference point onto successive simple outer approximations to the 
Kuhn-Tucker set of the inclusion $0\in Ax+L^*BLx$, lends itself 
to such an extension. Our analysis will borrow tools on 
Legendre functions and Bregman-based algorithms from \cite{Ccm01} 
and \cite{Sico03}, as well as geometric constructs from 
\cite{Nfao15} and \cite{Pams03}. The proposed results will provide
not only the first splitting methods for composite inclusions 
outside of Hilbert spaces, but also new algorithms in Hilbert, 
and even Euclidean, spaces.

The problem under consideration is the following.

\begin{problem}
\label{prob:1}
Let $\XX$ and $\YY$ be reflexive real Banach spaces such that
$\XX\neq\{0\}$ and $\YY\neq\{0\}$, let $\XXX$ be the standard 
product vector space $\XX\times\YY^*$ equipped with the norm 
$(x,y^*)\mapsto\sqrt{\|x\|^2+\|y^*\|^2}$, and let $\XXX^*$ be its 
topological dual, that is, $\XX^*\times\YY$ equipped with 
the norm $(x^*,y)\mapsto\sqrt{\|x^*\|^2+\|y\|^2}$.
Let $A\colon\XX\to 2^{\XX^*}$ and 
$B\colon\YY\to 2^{\YY^*}$ be maximally monotone, and
let $L\colon\XX\to\YY$ be linear and bounded. Consider the
inclusion problem
\begin{equation}
\label{Ipb:Ia}
\text{find}\;\;x\in\XX\;\;\text{such that}\;\;0\in Ax+L^*BLx,
\end{equation}
the dual problem
\begin{equation}
\label{Ipb:Ib}
\text{find}\;\;y^*\in\YY^*\;\;\text{such that}\;\; 
0\in -LA^{-1}(-L^*y^*)+B^{-1}y^*,
\end{equation}
and let 
\begin{equation}
\label{IVe:VKT}
\boldsymbol{Z}
=\menge{(x,y^*)\in\boldsymbol{\mathcal{X}}}{-L^*y^*\in Ax
\quad\text{and}\quad Lx\in B^{-1}y^*}
\end{equation}
be the associated Kuhn-Tucker set.
Let $f\colon\XX\to\RX$ and $g\colon\YY\to\RX$ be Legendre 
functions, set
\begin{equation}
\label{epOlk28ha14}
\boldsymbol{f}\colon\XXX\to\RX\colon(x,y^*)\mapsto f(x)+g^*(y^*),
\end{equation}
let $x_0\in\intdom f$, let $y_0^*\in\intdom g^*$, 
and suppose that 
$\boldsymbol{Z}\cap\text{\rm int\,dom}\,\boldsymbol{f}\neq\emp$.
The problem is to find the best Bregman approximation 
$(\overline{x},\overline{y}^*)=P^{\boldsymbol{f}}_{\boldsymbol{Z}}
(x_0,y_0^*)$ to $(x_0,y_0^*)$ from $\boldsymbol{Z}$.
\end{problem}

{\bfseries Notation.} 
The symbols $\weakly$ and $\to$ denote respectively weak and 
strong convergence. The set of weak sequential cluster points 
of a sequence $(x_n)_{n\in\NN}$ is denoted by $\WC(x_n)_{n\in\NN}$.
The closed ball of center $x\in\XX$ and radius $\rho\in\RPP$ is
denoted by $B(x;\rho)$.
Let $M\colon\XX\to 2^{\XX^*}$ be a set-valued operator. The domain 
of $M$ is $\dom M=\menge{x\in\XX}{Mx\neq\emp}$, the range of 
$M$ is $\ran M=\menge{x^*\in\XX^*}{(\exi x\in\XX)\,x^*\in Mx}$, 
and the set of zeros of $M$ is $\zer M=\menge{x\in\XX}{0\in Mx}$. 
$\Gamma_0(\XX)$ is the class of all lower semicontinuous convex 
functions $f\colon\XX\to\RX$ such that 
$\dom f=\menge{x\in\XX}{f(x)<\pinf}\neq\emp$. 
Let $f\colon\XX\to\RX$. Then $f$ is coercive if 
$\lim_{\|x\|\to+\infty}f(x)=+\infty$ and supercoercive if 
$\lim_{\|x\|\to+\infty}f(x)/\|x\|=+\infty$.

\section{Preliminary results}
\label{IVsec:II}

\subsection{Properties of the Kuhn-Tucker set}
The following proposition revisits and complements some results 
of \cite{Siop14} and \cite{Siop11} on the properties of the 
Kuhn-Tucker set in the more general setting of 
Problem~\ref{prob:1}.

\begin{proposition}
\label{ppOlk28ha12}
Consider the setting of Problem~\ref{prob:1}. Then the 
following hold:
\begin{enumerate}
\item
\label{ppOlk28ha12i}
Let $\mathscr{P}$ be the set of solutions to \eqref{Ipb:Ia} and 
let $\mathscr{D}$ be the set of solutions to \eqref{Ipb:Ib}. Then 
the following hold:
\begin{enumerate}
\item
\label{ppOlk28ha12ia}
$\boldsymbol{Z}$ is a closed convex subset of 
$\mathscr{P}\times\mathscr{D}$.
\item
\label{ppOlk28ha12ib}
Set $\mathcal{Q}_\XX\colon\XXX\to\XX\colon (x,y^*)\mapsto x$
and $\mathcal{Q}_{\YY^*}\colon\XXX\to\YY^*\colon (x,y^*)\mapsto y^*$.
Then $\mathscr{P}=\mathcal{Q}_\XX(\boldsymbol{Z})$ and 
$\mathscr{D}=\mathcal{Q}_{\YY^*}(\boldsymbol{Z})$. 
\item
\label{ppOlk28ha12ic}
$\mathscr{P}\neq\emp\Leftrightarrow\boldsymbol{Z}
\neq\emp\Leftrightarrow\mathscr{D}\neq\emp$.
\end{enumerate}
\item
\label{ppOlk28ha12ii}
For every $\mathsf{a}=(a,a^*)\in\gra A$ and 
$\mathsf{b}=(b,b^*)\in\gra B$, set
$\boldsymbol{s}_{\mathsf{a},\mathsf{b}}^*=(a^*+L^*b^*,b-La)$, 
$\eta_{\mathsf{a},\mathsf{b}}=\pair{a}{a^*}+\pair{b}{b^*}$, and 
$\boldsymbol{H}_{\mathsf{a},\mathsf{b}}=
\menge{\boldsymbol{x}\in\boldsymbol{\mathcal{X}}}
{\pair{\boldsymbol{x}}{\boldsymbol{s}_{\mathsf{a},\mathsf{b}}^*}
\leq\eta_{\mathsf{a},\mathsf{b}}}$. Then the following hold:
\begin{enumerate}
\item
\label{ppOlk28ha12iia}
$(\forall\mathsf{a}\in\gra A)(\forall\mathsf{b}\in\gra B)$ 
$[\,\boldsymbol{s}^*_{\mathsf{a},\mathsf{b}}
=\boldsymbol{0}\;\Leftrightarrow\;
\boldsymbol{H}_{\mathsf{a},\mathsf{b}}
=\XXX\;\Rightarrow\;(a,b^*)\in\boldsymbol{Z}
\;\text{and}\;\eta_{\mathsf{a},\mathsf{b}}=0\,]$.
\item
\label{ppOlk28ha12iib}
$\boldsymbol{Z}=\bigcap_{\mathsf{a}\in\gra A}
\bigcap_{\mathsf{b}\in\gra B}
\boldsymbol{H}_{\mathsf{a},\mathsf{b}}$\,.
\end{enumerate}
\item
\label{ppOlk28ha12iii}
Let $(a_n,a_n^*)_{n\in\NN}$ be a sequence 
in $\gra A$, let $(b_n,b_n^*)_{n\in\NN}$ be a sequence in 
$\gra B$, and let $(x,y^*)\in\XXX$. 
Suppose that $a_n\weakly x$, $b_n^*\weakly y^*$, 
$a_n^*+L^*b_n^*\to 0$, and $La_n-b_n\to 0$. Then 
$(x,y^*)\in\boldsymbol{Z}$.
\end{enumerate}
\end{proposition}
\begin{proof}
Set $\boldsymbol{M}\colon\boldsymbol{\mathcal{X}}\to
2^{\boldsymbol{\mathcal{X}}^*}\colon(x,y^*)\mapsto 
Ax\times B^{-1}y^*$ and
$\boldsymbol{S}\colon\boldsymbol{\mathcal{X}}
\to\boldsymbol{\mathcal{X}}^*\colon (x,y^*)\mapsto (L^*y^*,-Lx)$.
Since $A$ and $B^{-1}$ are maximally monotone, so is 
$\boldsymbol{M}$. On the other hand, $\boldsymbol{S}$ is linear, 
bounded, and positive since 
\begin{equation}
\label{egVxZ5y821}
(\forall(x,y^*)\in\XXX)\quad
\pair{\boldsymbol{S}(x,y^*)}{(x,y^*)}=
\pair{x}{L^*y^*}+\pair{-Lx}{y^*}=0. 
\end{equation}
Thus, it follows from \cite[Section~17]{Simo08} that 
$\boldsymbol{S}$ is maximally monotone with 
$\dom\boldsymbol{S}=\XXX$. In turn, we derive from 
\cite[Theorem~24.1(a)]{Simo08} that
\begin{equation}
\label{egVxZ5y820}
\boldsymbol{M}+\boldsymbol{S}\;\text{is maximally monotone.}
\end{equation}

\ref{ppOlk28ha12ia}:
Let $(x,y^*)\in\XXX$. Then
\begin{eqnarray}
\label{egVxZ5y828}
\boldsymbol{0}\in\boldsymbol{M}(x,y^*)+\boldsymbol{S}(x,y^*)
&\Leftrightarrow& 0\in Ax+L^*y^*\quad\text{and}\quad 0\in 
B^{-1}y^*-Lx\nonumber\\
&\Leftrightarrow& -L^*y^*\in Ax
\quad\text{and}\quad Lx\in B^{-1}y^*\nonumber\\
&\Leftrightarrow& (x,y^*)\in \boldsymbol{Z}.
\end{eqnarray}
Therefore, we derive from \eqref{egVxZ5y820} and
\cite[Lemma~1.1(a)]{IVBro65} that
$\boldsymbol{Z}=\zer(\boldsymbol{M}+\boldsymbol{S})=
(\boldsymbol{M}+\boldsymbol{S})^{-1}(0)$ is closed and convex.

\ref{ppOlk28ha12ib}: Let $x\in\XX$. Then
$x\in\mathscr{P}\Leftrightarrow
0\in Ax+L^*BLx\Leftrightarrow (\exi y^*\in\YY^*)\; 
[-L^*y^*\in Ax\;\text{and}\;y^*\in BLx]
\Leftrightarrow (\exi y^*\in\YY^*)\;(x,y^*)
\in\boldsymbol{Z}$. 
Hence $\mathscr{P}\neq\emp$
$\Leftrightarrow$ $\boldsymbol{Z}\neq\emp$.
Likewise, let $y^*\in\YY^*$. Then 
$y^*\in\mathscr{D}\Leftrightarrow
0\in-LA^{-1}(-L^*y^*)+B^{-1}y^*
\Leftrightarrow(\exi x\in\XX)\; 
[x\in A^{-1}(-L^*y^*)\;\text{and}\;0\in-Lx+B^{-1}y^*]
\Leftrightarrow(\exi x\in\XX)\; 
[-L^*y^*\in Ax\;\text{and}\; Lx\in B^{-1}y^*]
\Leftrightarrow(\exi x\in\XX)\;(x,y^*)\in\boldsymbol{Z}$.

\ref{ppOlk28ha12ic}: Clear from \ref{ppOlk28ha12ib} (see also
\cite[Corollary 2.4]{IVPe00}).

\ref{ppOlk28ha12iia} :
Let $\mathsf{a}\in\gra A$ and $\mathsf{b}\in\gra B$. Then
$\boldsymbol{s}^*_{\mathsf{a},\mathsf{b}}=\boldsymbol{0}$
$\Rightarrow$ 
[$-L^*b^*=a^*\in Aa$ and $La=b\in B^{-1}b^*$] 
$\Rightarrow$ $(a,b^*)\in\boldsymbol{Z}$. In 
addition, 
\begin{multline}
\boldsymbol{s}^*_{\mathsf{a},\mathsf{b}}=\boldsymbol{0}
\quad\Rightarrow\quad
\eta_{\mathsf{a},\mathsf{b}}=\pair{a}{a^*}
+\pair{b}{b^*}=\pair{a}{-L^*b^*}+\pair{La}{b^*}\\
=-\pair{La}{b^*}+\pair{La}{b^*}=0.
\end{multline}
Thus 
$\boldsymbol{s}^*_{\mathsf{a},\mathsf{b}}=\boldsymbol{0}$
$\Rightarrow$ $\boldsymbol{H}_{\mathsf{a},\mathsf{b}}=\XXX$. 
Conversely, 
$\boldsymbol{H}_{\mathsf{a},\mathsf{b}}=\XXX$ $\Rightarrow$ 
$\boldsymbol{s}^*_{\mathsf{a},\mathsf{b}}=\boldsymbol{0}$
$\Rightarrow$ $\eta_{\mathsf{a},\mathsf{b}}=0$.

\ref{ppOlk28ha12iib}: First, suppose that
$\boldsymbol{x}=(x,y^*)\in\bigcap_{\mathsf{a}\in\gra A}
\bigcap_{\mathsf{b}\in\gra B}
\boldsymbol{H}_{\mathsf{a},\mathsf{b}}$. Then 
\begin{align}
\label{epOlk28ha11}
&\hskip -2mm
(\forall\mathsf{a}\in\gra A)(\forall\mathsf{b}\in\gra B)\quad
\pair{(a,b^*)-(x,y^*)}{(a^*,b)-(-L^*y^*,Lx)}\nonumber\\
&\hskip 60mm =\pair{(a-x,b^*-y^*)}{(a^*+L^*y^*,b-Lx)}\nonumber\\
&\hskip 60mm =\pair{a-x}{a^*+L^*y^*}+\pair{b-Lx}{b^*-y^*}\nonumber\\
&\hskip 60mm =\eta_{\mathsf{a},\mathsf{b}}-\pair{\boldsymbol{x}}
{\boldsymbol{s}^*_{\mathsf{a},\mathsf{b}}}\nonumber\\
&\hskip 60mm \geq 0.
\end{align}
On the other hand, since 
\begin{equation}
\label{epOlk28ha18x}
\menge{\big((a,b^*),(a^*,b)\big)}{\mathsf{a}\in\gra A, 
\mathsf{b}\in\gra B}=\gra\boldsymbol{M}, 
\end{equation}
it follows from \eqref{egVxZ5y820} and \eqref{epOlk28ha11} that
$((x,y^*),(-L^*y^*,Lx))\in\gra\boldsymbol{M}$, i.e., 
$\boldsymbol{x}\in\boldsymbol{Z}$. Thus 
\begin{equation}
\label{epOlk28ha19}
\bigcap_{\mathsf{a}\in\gra A}\bigcap_{\mathsf{b}\in\gra B}
\boldsymbol{H}_{\mathsf{a},\mathsf{b}}\subset\boldsymbol{Z}.
\end{equation}
Conversely, let $\mathsf{a}\in\gra A$, let $\mathsf{b}\in\gra B$,
and let $(x,y^*)\in\boldsymbol{Z}$. 
Then $(x,-L^*y^*)\in\gra A$ and $(Lx,y^*)\in\gra B$. Since
$A$ and $B$ are monotone, we obtain
\begin{equation}
\pair{a-x}{a^*+L^*y^*}\geq 0
\quad\text{and}\quad
\pair{b-Lx}{b^*-y^*}\geq 0.
\end{equation}
Adding these two inequalities yields
\begin{equation}
\pair{x-a}{a^*+L^*y^*}+\pair{Lx-b}{b^*-y^*}\leq 0
\end{equation}
and, therefore,
\begin{align}
\label{epOlk28ha18}
\pair{\boldsymbol{x}}{\boldsymbol{s}^*_{\mathsf{a},\mathsf{b}}}
&=\pair{x}{a^*+L^*b^*}+\pair{b-La}{y^*}\nonumber\\
&=\pair{x}{a^*+L^*y^*}+\pair{Lx}{b^*-y^*}
+\pair{b-Lx}{y^*}+\pair{x-a}{L^*y^*}\nonumber\\
&=\pair{x-a}{a^*+L^*y^*}+\pair{a}{a^*}+\pair{La}{y^*}\nonumber\\
&\quad\;+\pair{Lx-b}{b^*-y^*}+\pair{b}{b^*}-\pair{b}{y^*}
+\pair{b-Lx}{y^*}+\pair{x-a}{L^*y^*}\nonumber\\
&\leq\pair{a}{a^*}+\pair{b}{b^*}+\pair{La-b}{y^*}
+\pair{b-Lx}{y^*}+\pair{x-a}{L^*y^*}\nonumber\\
&=\pair{a}{a^*}+\pair{b}{b^*}\nonumber\\
&=\eta_{\mathsf{a},\mathsf{b}}.
\end{align}
This implies that 
$(x,y^*)\in\boldsymbol{H}_{\mathsf{a},\mathsf{b}}$. Hence, 
$\boldsymbol{Z}\subset\boldsymbol{H}_{\mathsf{a},\mathsf{b}}$. 

\ref{ppOlk28ha12iii}: 
Set $(\forall n\in\NN)$ $\boldsymbol{x}_n=(a_n,b_n^*)$ and 
$\boldsymbol{x}_n^*=(a_n^*+L^*b_n^*,b_n-La_n)$. Then 
$\boldsymbol{x}_n\weakly (x,y^*)$, 
$\boldsymbol{x}_n^*\to \boldsymbol{0}$, and $(\forall n\in\NN)$ 
$(\boldsymbol{x}_n,\boldsymbol{x}_n^*)\in
\gra(\boldsymbol{M}+\boldsymbol{S})$. 
However, it follows from \eqref{egVxZ5y820} that 
$\gra(\boldsymbol{M}+\boldsymbol{S})$ is sequentially closed 
in $\XXX^{\text{weak}}\times\XXX^{*\text{strong}}$ 
\cite[Lemma~1.2]{IVBro65}. Hence, 
$\boldsymbol{0}\in(\boldsymbol{M}+\boldsymbol{S})(x,y^*)$, 
i.e., by \eqref{egVxZ5y828}, $(x,y^*)\in\boldsymbol{Z}$.
\end{proof}

\begin{proposition}
\label{pvb8uh4504}
Consider the setting of Problem~\ref{prob:1}. Then the 
following hold:
\begin{enumerate}
\item
\label{pvb8uh4504i}
$\boldsymbol{f}$ is a Legendre function.
\item
\label{pvb8uh4504ii}
The solution $(\overline{x},\overline{y}^*)$ to
Problem~\ref{prob:1} exists and is unique.
\end{enumerate}
\end{proposition}
\begin{proof}
\ref{pvb8uh4504i}:
Since $f$ and $g$ are Legendre functions, so are $f^*$ and $g^*$  
\cite[Corollary~5.5]{Ccm01}. Therefore, it follows from 
\cite[Theorem~5.6(iii)]{Ccm01} that $\partial f$ and 
$\partial g^*$ are single-valued on $\dom\partial f=\intdom f$ 
and $\dom\partial g^*=\intdom g^*$, respectively. 
On the other hand, we derive from \eqref{epOlk28ha14} that 
$\dom\boldsymbol{f}=\dom f\times\dom g^*$ and that
\begin{equation}
\partial\boldsymbol{f}\colon\boldsymbol{\XX}\to 
2^{\boldsymbol{\XX^*}}\colon (x,y^*)\mapsto 
\partial f(x)\times\partial g^*(y^*).
\end{equation}
Thus, $\partial\boldsymbol{f}$ is single-valued on 
\begin{equation}
\label{ebcn5y731}
\dom\partial\boldsymbol{f}=\dom\partial f\times\dom\partial g^*
=\intdom f\times\intdom g^*=\inte(\dom f\times\dom g^*)
=\intdom\boldsymbol{f}.
\end{equation}
Likewise, since 
\begin{equation}
\boldsymbol{f}^*\colon\XX^*\to\RX\colon 
(x^*,y)\mapsto f^*(x^*)+g(y), 
\end{equation}
we deduce that $\partial\boldsymbol{f}^*$ is single-valued on 
$\dom\partial\boldsymbol{f}^*=\intdom\boldsymbol{f}^*$. 
Consequently, \cite[Theorems~5.4 and~5.6]{Ccm01} assert that 
$\boldsymbol{f}$ is a Legendre function.

\ref{pvb8uh4504ii}: It follows from 
Proposition~\ref{ppOlk28ha12}\ref{ppOlk28ha12ia} that 
$\boldsymbol{Z}$ is a closed convex subset of $\XXX$.
Hence, since
$\boldsymbol{Z}\cap\text{\rm int\,dom}\,\boldsymbol{f}\neq\emp$,
we derive from \ref{pvb8uh4504i} and 
\cite[Corollary~7.9]{Ccm01} that
$(\overline{x},\overline{y}^*)=P^{\boldsymbol{f}}_{\boldsymbol{Z}}
(x_0,y_0^*)$ is uniquely defined.
\end{proof}

\subsection{Best Bregman approximation algorithm}
The approach we present goes back to Haugazeau's algorithm 
\cite[Th\'eor\`eme~3-2]{Haug68} (see also
\cite[Theorem~29.3]{Livre1}) for projecting a point onto 
the intersection of closed convex sets in a Hilbert space using 
the projections onto the individual sets. The method was 
extended in \cite{Sico00} to minimize certain convex functions 
over the intersection of closed convex sets in Banach spaces. 
The adaptation to the problem of finding the best Bregman
approximation from a closed convex set was investigated in
\cite{Pams03}.

\begin{definition}
\label{d:B-haugazeau}
{\rm\cite[Definition~3.1]{Sico03} and \cite[Section~3]{Pams03}}
Let $\XXX$ be a reflexive real Banach space, let 
$\boldsymbol{f}\in\Gamma_0(\XXX)$ be a Legendre function, let 
$\boldsymbol{x}_0\in\intdom\boldsymbol{f}$, let 
$\boldsymbol{x}\in\intdom\boldsymbol{f}$, and let 
$\boldsymbol{y}\in\intdom\boldsymbol{f}$. Then
\begin{align}
\label{eKKnhf74g19a}
H^{\boldsymbol{f}}(\boldsymbol{x},\boldsymbol{y})
&=\menge{\boldsymbol{z}\in\boldsymbol{\XX}}{\pair{\boldsymbol{z}
-\boldsymbol{y}}{\nabla\boldsymbol{f}(\boldsymbol{x})
-\nabla\boldsymbol{f}(\boldsymbol{y})}\leq 0}\nonumber\\
&=\menge{\boldsymbol{z}\in\boldsymbol{\XX}}{
D^{\boldsymbol{f}}(\boldsymbol{z},\boldsymbol{y})+
D^{\boldsymbol{f}}(\boldsymbol{y},\boldsymbol{x})\leq
D^{\boldsymbol{f}}(\boldsymbol{z},\boldsymbol{x})}
\end{align}
is the closed affine half-space onto which $\boldsymbol{y}$ 
is the Bregman projection of $\boldsymbol{x}$ if
$\boldsymbol{x}\neq \boldsymbol{y}$. 
Moreover, if $H^{\boldsymbol{f}}(\boldsymbol{x}_0,\boldsymbol{x})
\cap H^{\boldsymbol{f}}(\boldsymbol{x},\boldsymbol{y})
\cap\intdom\boldsymbol{f}\neq\emp$, then
\begin{equation}
\label{e:Q}
Q^{\boldsymbol{f}}(\boldsymbol{x}_0,\boldsymbol{x},\boldsymbol{y})
=P_{H^{\boldsymbol{f}}(\boldsymbol{x}_0,\boldsymbol{x})\cap 
H^{\boldsymbol{f}}(\boldsymbol{x},\boldsymbol{y})}^{\boldsymbol{f}}
\boldsymbol{x}_0.
\end{equation}
\end{definition}

\begin{lemma} {\rm\cite[Lemma~3.2]{Sico03}}
\label{l:2002}
Let $\XXX$ be a reflexive real Banach space, and let 
$\boldsymbol{C}_1$ and $\boldsymbol{C}_2$ be convex subsets
of $\XXX$ such that $\boldsymbol{C}_1$ is closed and 
$\boldsymbol{C}_1\cap\inte\boldsymbol{C}_2\neq\emp$. Then 
$\overline{\boldsymbol{C}_1\cap\inte\boldsymbol{C}}_2=
\boldsymbol{C}_1\cap\overline{\boldsymbol{C}}_2$.
\end{lemma}

\begin{proposition}
\label{p:2}
Let $\XXX$ be a reflexive real Banach space, let 
$\boldsymbol{f}\in\Gamma_0(\XXX)$ be a Legendre 
function, let $\boldsymbol{C}$ be a closed convex subset 
of $\overline{\dom}\boldsymbol{f}$ such that 
$\boldsymbol{C}\cap\intdom\boldsymbol{f}\neq\emp$,
let $\boldsymbol{x}_0\in\intdom\boldsymbol{f}$, and set 
$\overline{\boldsymbol{x}}=P_{\boldsymbol{C}}^{\boldsymbol{f}}
\boldsymbol{x}_0$. 
At every iteration $n\in\NN$, find 
$\boldsymbol{x}_{n+1/2}\in\inte\dom\boldsymbol{f}$ such that 
$\boldsymbol{C}\subset H^{\boldsymbol{f}}
(\boldsymbol{x}_n,\boldsymbol{x}_{n+1/2})$ and set
\begin{equation}
\label{e:haugazeau}
\boldsymbol{x}_{n+1}=Q^{\boldsymbol{f}}
\big(\boldsymbol{x}_0,\boldsymbol{x}_n,\boldsymbol{x}_{n+1/2}\big).
\end{equation}
Then the following hold:
\begin{enumerate}
\item
\label{p:2i-}
$(\forall n\in\NN)$ $\boldsymbol{C}\subset H^{\boldsymbol{f}}
(\boldsymbol{x}_0,\boldsymbol{x}_n)$.
\item
\label{p:2i}
$(\boldsymbol{x}_n)_{n\in\NN}$ is a well-defined bounded sequence 
in $\intdom\boldsymbol{f}$.
\item
\label{p:2ii+} 
Suppose that, for some $n\in\NN$, 
$\boldsymbol{x}_n\in\boldsymbol{C}$. Then $(\forall k\in\NN)$ 
$\boldsymbol{x}_{n+k}=\overline{\boldsymbol{x}}$.
\item
\label{p:2ii} 
$\sum_{n\in\NN}D^{\boldsymbol{f}}(\boldsymbol{x}_{n+1},
\boldsymbol{x}_n)<+\infty$.
\item
\label{p:2iii-} 
$(\forall n\in\NN)$ $D^{\boldsymbol{f}}
(\boldsymbol{x}_{n+1/2},\boldsymbol{x}_n)\leq 
D^{\boldsymbol{f}}(\boldsymbol{x}_{n+1},\boldsymbol{x}_n)$.
\item
\label{p:2iii} 
$\sum_{n\in\NN}D^{\boldsymbol{f}}(\boldsymbol{x}_{n+1/2},
\boldsymbol{x}_n)<+\infty$.
\item
\label{p:2iv} 
$[\,\boldsymbol{x}_n\weakly\overline{\boldsymbol{x}}$ and 
$\boldsymbol{f}(\boldsymbol{x}_n)\to\boldsymbol{f}
(\overline{\boldsymbol{x}})\,]$ 
$\Leftrightarrow$ $D^{\boldsymbol{f}}(\boldsymbol{x}_n,
\overline{\boldsymbol{x}})\to 0$ 
$\Leftrightarrow$
$\WC(\boldsymbol{x}_n)_{n\in\NN}\subset\boldsymbol{C}$. 
\end{enumerate}
\end{proposition}
\begin{proof}
Item~\ref{p:2i-} is found in 
\cite[Proof of Proposition~3.3]{Pams03}.
The first equivalence in \ref{p:2iv} follows from
\cite[Propositions~2.2(ii)]{Pams03}. To establish the remaining
assertions, set
\begin{equation}
\label{epOlk28ha21a}
(\forall n\in\NN)\quad T_n=
P_{H^{\boldsymbol{f}}
(\boldsymbol{x}_n,\boldsymbol{x}_{n+1/2})}^{\boldsymbol{f}}.
\end{equation}
Then, for every $n\in\NN$, Definition~\ref{d:B-haugazeau}
yields
$T_n\boldsymbol{x}_n=\boldsymbol{x}_{n+1/2}$, 
\cite[Proposition~3.32(ii)(b)]{Sico03} yields
$\Fix T_n=H^{\boldsymbol{f}}(\boldsymbol{x}_n,\boldsymbol{x}_{n+1/2})
\cap\IDDB$, and we derive from Lemma~\ref{l:2002} that
\begin{equation}
\label{epOlk28ha21b}
\boldsymbol{C}\subset H^{\boldsymbol{f}}
(\boldsymbol{x}_n,\boldsymbol{x}_{n+1/2})
\cap\overline{\dom}\boldsymbol{f}=
\overline{H^{\boldsymbol{f}}
(\boldsymbol{x}_n,\boldsymbol{x}_{n+1/2})\cap\IDDB}
=\overline{\Fix}T_n.
\end{equation}
On the other hand,
\begin{equation}
\label{epOlk28ha21c1}
(\forall n\in\NN)\quad\emp\neq \boldsymbol{C}\cap\intdom
\boldsymbol{f}\subset
H^{\boldsymbol{f}}(\boldsymbol{x}_n,\boldsymbol{x}_{n+1/2})
\cap\intdom\boldsymbol{f}=\Fix T_n
\end{equation}
and, therefore, $\bigcap_{n\in\NN}\Fix T_n\neq\emp$.
Altogether, \cite[Condition~3.2]{Pams03} is satisfied
and it follows from \cite[Propositions~3.3 and 3.4]{Pams03} 
and \cite[Proof of Proposition~3.4(vii)]{Pams03} 
that the proof is complete.
\end{proof}

\subsection{Coercivity and boundedness of monotone operators}

\begin{definition}
\label{d:coer-bou}
Let $\XX$ be a reflexive real Banach space such that 
$\XX\neq\{0\}$ and let $M\colon\XX\to 2^{\XX^*}$. 
Then $M$ is \emph{coercive} if 
\begin{equation}
\label{epOlk28ha21c}
(\exi z\in\dom M)\quad
\lim_{\|x\|\to+\infty}\inf\dfrac{\pair{x-z}{Mx}}{\|x\|}=\pinf,
\end{equation}
and it is \emph{bounded} if it maps bounded sets to bounded set. 
\end{definition}

\begin{lemma}
\label{l:20}
Let $\XX$ be a reflexive real Banach space such that 
$\XX\neq\{0\}$ and let $M\colon\XX\to 2^{\XX^*}$.
Suppose that one of the following holds:
\begin{enumerate}
\item
\label{l:20i}
$\dom M$ is nonempty and bounded.
\item
\label{l:20ii}
$M$ is uniformly monotone at some point $z\in\dom M$ with a
supercoercive modulus: there exists a strictly increasing function 
$\phi\colon\left[0,+\infty\right[\to\left[0,+\infty\right]$ that 
vanishes only at $0$ such that 
$\lim\limits_{t\to+\infty}\phi(t)/t=+\infty$ and
\begin{equation}
(\forall (x,x^*)\in\gra M)(\forall z^*\in Mz)\quad 
\pair{x-z}{x^*-z^*}\geq\phi(\|x-z\|).
\end{equation}
\item
\label{l:20iii}
$M=\partial\varphi$, where $\varphi$ is a supercoercive function 
in $\Gamma_0(\XX)$.
\end{enumerate}
Then $M$ is coercive. 
\end{lemma}
\begin{proof}
\ref{l:20i}:
Let $x\in\XX$ and let $z\in\dom M$. Then, if $\|x\|$ is
sufficiently large, we have $Mx=\emp$ and therefore 
$\inf{\pair{x-z}{Mx}}/{\|x\|}=\pinf$.

\ref{l:20ii}: We have
\begin{equation}
(\forall (x,x^*)\in\gra M)(\forall z^*\in Mz)\quad 
\pair{x-z}{x^*-z^*}\geq\phi(\|x-z\|).
\end{equation}
Hence, for every $x\in\dom M$ such that $\|x\|>\|z\|$, we have
\begin{align}
(\forall x^*\in Mx)(\forall z^*\in Mz)\quad
\dfrac{\pair{x-z}{x^*}}{\|x\|}
&\geq\dfrac{\phi(\|x-z\|)-\|x-z\|\,\|z^*\|}{\|x\|}\nonumber\\
&=\dfrac{\|x-z\|}{\|x\|}\bigg(
\dfrac{\phi(\|x-z\|)}{\|x-z\|}-\|z^*\|\bigg).
\end{align}
Thus, 
\begin{equation}
\lim_{\|x\|\to+\infty}\inf\dfrac{\pair{x-z}{Mx}}{\|x\|}=\pinf.
\end{equation}

\ref{l:20iii}: In view of \ref{l:20i}, we suppose that
$\dom M$ is unbounded. Let $z\in\dom M$. Then we derive from 
\eqref{e:subdiff} that, for every 
$x\in\dom M\smallsetminus\{0\}$,
\begin{equation}
\label{eQA8h34-05}
(\forall x^*\in Mx)\quad\dfrac{\varphi(x)-\varphi(z)}{\|x\|}\leq
\dfrac{\pair{x-z}{x^*}}{\|x\|}.
\end{equation}
Hence, the supercoercivity of $\varphi$ yields
\begin{equation}
\label{eQA8h34-06}
\lim_{\|x\|\to+\infty}\inf\dfrac{\pair{x-z}{Mx}}{\|x\|}=\pinf
\end{equation}
and $M$ is therefore coercive.
\end{proof}

\begin{lemma}
\label{IVlhHg7yG02}
Let $\XX$ be a reflexive real Banach space such that 
$\XX\neq\{0\}$, let $M_1\colon\XX\to 2^{\XX^*}$, and let 
$M_2\colon\XX\to 2^{\XX^*}$ be monotone. Suppose that there 
exists $z\in\dom M_1\cap\dom M_2$ such that
\begin{equation}
\label{e:M1}
\lim\limits_{\|x\|\to+\infty}\inf\dfrac{\pair{x-z}{M_1x}}{\|x\|}
=\pinf.
\end{equation} 
Then $M_1+M_2$ is coercive.
\end{lemma}
\begin{proof}
Suppose that $x\in(\dom M_1\cap\dom M_2)\smallsetminus\{0\}$, let 
$x^*\in(M_1+M_2)x$, and let $z^*\in(M_1+M_2)z$. 
Then there exist $x_1^*\in M_1x$, 
$x_2^*\in M_2x$, $z_1^*\in M_1z$, and $z_2^*\in M_2z$ such that
$x^*=x_1^*+x_2^*$ and $z^*=z_1^*+z_2^*$. In turn, the monotonicity
of $M_2$ yields
\begin{align}
\dfrac{\pair{x-z}{x^*}}{\|x\|}
&=\dfrac{\pair{x-z}{x_1^*-z_1^*}}{\|x\|}
+\dfrac{\pair{x-z}{x_2^*-z_2^*}}{\|x\|}
+\dfrac{\pair{x-z}{z^*}}{\|x\|}\nonumber\\
&\geq\dfrac{\pair{x-z}{x_1^*}}{\|x\|}
+\dfrac{\pair{x-z}{z_2^*}}{\|x\|}\nonumber\\
&\geq\dfrac{\pair{x-z}{x_1^*}-\|x-z\|\,\|z_2^*\|}{\|x\|}
\end{align}
and \eqref{e:M1} implies that $M_1+M_2$ is coercive.
\end{proof}

\begin{lemma}
\label{IVlKKnhf74g12}
Let $\XX$ be a reflexive real Banach space such that 
$\XX\neq\{0\}$, let $M_1\colon\XX\to 2^{\XX^*}$, let 
$M_2\colon\XX\to 2^{\XX^*}$ be monotone, let $(x_n^*)_{n\in\NN}$ 
be a bounded sequence in $\XX^*$, and let $(\gamma_n)_{n\in\NN}$ 
be a bounded sequence in $\RPP$. Suppose that there exists 
$z\in\dom M_1\cap\dom M_2$ such that
\begin{equation}
\label{eKKnhf74g13b}
\lim_{\|x\|\to+\infty}\inf\dfrac{\pair{x-z}{M_1x}}{\|x\|}=\pinf,
\end{equation} 
and that
\begin{equation}
\label{eKKnhf74g12a}
(\forall n\in\NN)\quad x_n\in(M_1+\gamma_n M_2)^{-1}x_n^*.
\end{equation}
Then $(x_n)_{n\in\NN}$ is bounded.
\end{lemma}
\begin{proof}
Set $\beta=\sup_{n\in\NN}\|x_n^*\|$ and 
$\sigma=\sup_{n\in\NN}\gamma_n$. It follows from
\eqref{eKKnhf74g12a} that, for every $n\in\NN$, there exist
$a_n^*\in M_1x_n$ and $b_n^*\in M_2x_n$ such that
$x_n^*=a_n^*+\gamma_n b_n^*$. If $(x_n)_{n\in\NN}$ is  
unbounded, there exists a strictly increasing sequence 
$(k_n)_{n\in\NN}$ in $\NN$ such that $0<\|x_{k_n}\|\uparrow\pinf$.
Therefore, \eqref{eKKnhf74g13b} yields
\begin{equation}
\label{eKKnhf74g13a}
\lim_{n\to\pinf}
\dfrac{\pair{x_{k_n}-z}{a^*_{k_n}}}{\|x_{k_n}\|}=\pinf.
\end{equation} 
Now let $z^*\in M_2 z$. By monotonicity of $M_2$,
$(\forall n\in\NN)$ $\pair{x_{k_n}-z}{b_{k_n}^*-z^*}\geq 0$.
Hence, \eqref{eKKnhf74g13a} implies that
\begin{align}
\label{eKKnhf74g12b}
\beta\bigg(1+\dfrac{\|z\|}{\|x_{k_0}\|}\bigg)
&\geq\beta\bigg(1+\dfrac{\|z\|}{\|x_{k_n}\|}\bigg)\nonumber\\
&\geq\beta\dfrac{\|x_{k_n}-z\|}{\|x_{k_n}\|}\nonumber\\
&\geq\dfrac{\pair{x_{k_n}-z}{x_{k_n}^*}}{\|x_{k_n}\|}\nonumber\\
&=\dfrac{\pair{x_{k_n}-z}{a_{k_n}^*}}{\|x_{k_n}\|}
+\gamma_{k_n}\dfrac{\pair{x_{k_n}-z}{b_{k_n}^*-z^*}}{\|x_{k_n}\|}
+\gamma_{k_n}\dfrac{\pair{x_{k_n}-z}{z^*}}{\|x_{k_n}\|}\nonumber\\
&\geq\dfrac{\pair{x_{k_n}-z}{a_{k_n}^*}}{\|x_{k_n}\|}
-\sigma\dfrac{\|x_{k_n}-z\|\,\|z^*\|}{\|x_{k_n}\|}\nonumber\\
&\geq\dfrac{\pair{x_{k_n}-z}{a_{k_n}^*}}{\|x_{k_n}\|}
-\sigma\|z^*\|\bigg(1+\dfrac{\|z\|}{\|x_{k_0}\|}\bigg)\nonumber\\
&\to\pinf,
\end{align}
and we reach a contradiction.
\end{proof}

\begin{corollary}
\label{IVcKKnhf74g12}
Let $\XX$ be a reflexive real Banach space such that 
$\XX\neq\{0\}$ and let $M\colon\XX\to 2^{\XX^*}$ be coercive.
Then $M^{-1}$ is bounded.
\end{corollary}
\begin{proof}
Take $M_1=M$ and $M_2=0$ in Lemma~\ref{IVlKKnhf74g12}.
\end{proof}

\begin{proposition}
\label{IVpp:20a}
Let $\XX$ be a reflexive real Banach space such that 
$\XX\neq\{0\}$, let $h\in\Gamma_0(\XX)$ be essentially smooth, 
and let $M\colon\XX\to 2^{\XX^*}$ be such that 
$\dom M\cap\intdom h\neq\emp$. Suppose that 
one of the following holds:
\begin{enumerate}
\item
\label{IVcl:20b1}
$\dom M\cap\intdom h$ is bounded.
\item
\label{IVcl:20b2}
There exists $z\in\dom M\cap\intdom h$ such that
\begin{equation}
\lim_{\|x\|\to+\infty}\inf\dfrac{\pair{x-z}{Mx}}{\|x\|}
=\pinf.
\end{equation} 
\item
\label{IVcl:20b3}
$M$ is uniformly monotone at a point 
$z\in\dom M\cap\intdom h$ with a
supercoercive modulus.
\item
\label{IVcl:20b4}
$M$ is monotone and $h$ is supercoercive.
\item
\label{IVcl:20b5}
$M$ is monotone and $h$ is uniformly convex at a point 
$z\in\dom M\cap\intdom h$, i.e., there exists an increasing 
function 
$\phi\colon\left[0,+\infty\right[\to\left[0,+\infty\right]$ 
that vanishes only at $0$ such that 
\begin{equation}
(\forall y\in\dom h)(\forall\alpha\in\left]0,1\right[)\;\;
h(\alpha y+(1-\alpha)z)+\alpha(1-\alpha)\phi(\|y-z\|)
\leq\alpha h(y)+(1-\alpha)h(z).
\end{equation}
\end{enumerate}
Then $\nabla h+M$ is coercive. If, in addition, $M$ is 
maximally monotone, then $\dom(\nabla h+M)^{-1}=\XX^*$.
\end{proposition}
\begin{proof} 
We first observe that \cite[Theorem~18.7]{Simo08} and 
\cite[Theorem~5.6]{Ccm01} imply that $\nabla h$ is maximally 
monotone and that $\dom\nabla h=\intdom h$.

\ref{IVcl:20b1}: Lemma~\ref{l:20}\ref{l:20i}.

\ref{IVcl:20b2}: 
It follows from Lemma~\ref{IVlhHg7yG02} that
$\nabla h+M$ is coercive. 

\ref{IVcl:20b3}: Since $\nabla h+M$ is uniformly monotone at $z$ 
with a supercoercive modulus, the claim follows from 
Lemma~\ref{l:20}\ref{l:20ii}.

\ref{IVcl:20b4}: 
Let $z\in\dom M\cap\intdom h=\dom M\cap\dom\partial h$. 
Then we derive from 
\eqref{eQA8h34-06} that 
\begin{equation}
\lim_{\|x\|\to+\infty}\dfrac{\pair{x-z}{\nabla h(x)}}{\|x\|}=\pinf.
\end{equation}
Thus, $\nabla h$ satisfies \eqref{e:M1} and it follows from 
Lemma~\ref{IVlhHg7yG02} that $\nabla h+M$ is coercive. 

\ref{IVcl:20b5}: It follows from 
\cite[Definition~2.2 and Remark~2.8]{IVZali83} that $\nabla h$ is
uniformly monotone at $z$ with a supercoercive modulus. Hence, 
$\nabla h+M$ is likewise and Lemma~\ref{l:20}\ref{l:20ii} implies 
that $\nabla h+M$ is coercive. Alternatively, this is a special
case of \ref{IVcl:20b4}.

Finally, suppose that $M$ is maximally monotone. Then 
\cite[Theorem~24.1(a)]{Simo08} asserts that $\nabla h+M$ is 
maximally monotone. Consequently, since $\nabla h+M$ is 
coercive, it follows from 
\cite[Corollary~II-B.32.35]{ZeidXX} that
$\dom(\nabla h+M)^{-1}=\ran(\nabla h+M)=\XX^*$.
\end{proof}

\begin{lemma}
\label{IVlKKnhf74g10}
Let $\XX$ and $\YY$ be real Banach spaces, let $D\subset\XX$
be a nonempty open set, and let $C$ be a nonempty bounded convex 
subset of $D$. Suppose that $T\colon D\to\YY$ is uniformly 
continuous on $C$ in the sense that 
\begin{equation}
\label{eKKnhf74g15t}
(\forall\varepsilon\in\RPP)(\exi\delta\in\RPP)(\forall x\in C)
(\forall y\in C)\quad
\|x-y\|\leq\delta\;\Rightarrow\;
\|Tx-Ty\|\leq\varepsilon.
\end{equation}
Then $T$ is bounded on $C$.
\end{lemma}
\begin{proof}
In view of \eqref{eKKnhf74g15t}, there exists
$\delta\in\RPP$ such that
\begin{equation}
\label{4:e31}
(\forall x\in C)(\forall y\in C)\quad 
\|x-y\|\leq\delta\quad\Rightarrow\quad\|Tx-Ty\|\leq 1.
\end{equation}
Now fix $z\in C$ and $\rho\in\RPP$ such that 
$C\subset\menge{x\in\XX}{\|x-z\|\leq\rho}$, and take
an integer $m\geq 1+\rho/\delta$. 
Let $x\in C$ and set 
\begin{equation}
(\forall n\in\{0,\ldots,m\})\quad  
x_n=x+\dfrac{n}{m}(z-x)\in C.
\end{equation}
Then, for every $n\in\{0,\ldots,m-1\}$,
$\|x_{n+1}-x_n\|=\|z-x\|/m\leq\rho/m\leq\delta$ and
\eqref{4:e31} yields $\|Tx_{n+1}-Tx_n\|\leq 1$. Hence,
$\|Tz-Tx\|\leq\sum_{n=0}^{m-1}\|Tx_{n+1}-Tx_n\|\leq m$ and
therefore $\|Tx\|<\|Tz\|+m$. We conclude that 
$\sup_{x\in C}\|Tx\|\leq\|Tz\|+m$.
\end{proof}

\section{Best Bregman approximation algorithm}
\label{IVsec:V}
Proposition~\ref{ppOlk28ha12}\ref{ppOlk28ha12ia} asserts that 
Problem~\ref{prob:1} reduces to finding the Bregman projection 
of a reference point $(x_0,y_0^*)$ onto the closed convex subset
$\boldsymbol{C}=\boldsymbol{Z}\cap\overline{\dom}\boldsymbol{f}$
of $\overline{\dom}\boldsymbol{f}$. Our strategy is to employ 
Proposition~\ref{p:2} for this task.
The following condition will be used subsequently
(see \cite[Examples~4.10, 5.11, and 5.13]{Sico03} for special cases).

\begin{condition}
{\rm\cite[Condition~4.3(ii)]{Pams03}}
\label{IVC:1}
Let $\XX$ be a reflexive real Banach space and let
$h\in\Gamma_0(\XX)$ be G\^ateaux differentiable on 
$\intdom h\neq\emp$. For every sequence 
$(x_n)_{n\in\NN}$ in $\intdom h$ and every bounded sequence 
$(y_n)_{n\in\NN}$ in $\intdom h$, 
\begin{equation}
\label{IVbcn5y713b}
D^h(x_n,y_n)\to 0\quad\Rightarrow\quad x_n-y_n\to 0.
\end{equation}
\end{condition}

We now derive from Proposition~\ref{p:2} our best Bregman
approximation algorithm to solve Problem~\ref{prob:1}.

\begin{theorem}
\label{IVt:1}
Consider the setting of Problem~\ref{prob:1}. Let 
$h\in\Gamma_0(\XX)$ and $j\in\Gamma_0(\YY)$ be Legendre 
functions such that $\intdom f\subset\intdom h$, 
$L(\intdom f)\subset\intdom j$, and there exist
$\varepsilon$ and $\delta$ in $\RPP$ such that 
$\nabla h+\varepsilon A$ and $\nabla j+\delta B$ are coercive. 
Let $\sigma\in\left[\max\{\varepsilon,\delta\},\pinf\right[$ 
and iterate
\begin{equation}
\label{IVe:Bfg+uh4f9-09h}
\begin{array}{l}
\text{for}\;n=0,1,\ldots\\
\left\lfloor
\begin{array}{l}
(\gamma_n,\mu_n)\in\left[\varepsilon,\sigma\right]
\times\left[\delta,\sigma\right]\\[1mm]
a_n=(\nabla h+\gamma_nA)^{-1}\big(\nabla h(x_n)-\gamma_n 
L^*y_n^*\big)\\[1mm]
a_n^*=\gamma_n^{-1}\big(\nabla h(x_n)-\nabla h(a_n)\big)-L^*y_n^*
\\[1mm]
b_n=(\nabla j+\mu_nB)^{-1}\big(\nabla j(Lx_n)+\mu_n y_n^*\big)
\\[1mm]
b_n^*=\mu_n^{-1}\big(\nabla j(Lx_n)-\nabla j(b_n)\big)+y_n^*
\\[1mm]
\boldsymbol{H}_n=
\menge{(x,y^*)\in\XXX}{\pair{x}{a_n^*+L^*b_n^*}
+\pair{b_n-La_n}{y^*}\leq\pair{a_n}{a_n^*}+\pair{b_n}{b_n^*}}\\
(x_{n+1/2},y_{n+1/2}^*)=
P_{\boldsymbol{H}_n}^{\boldsymbol{f}}
(x_n,y_n^*)\\
(x_{n+1},y_{n+1}^*)=Q^{\boldsymbol{f}}
\big((x_0,y_0^*),(x_n,y_n^*),(x_{n+1/2},y_{n+1/2}^*)\big).
\end{array}
\right.\\
\end{array}
\end{equation}
Then the following hold:
\begin{enumerate}
\item
\label{IVt:1i} 
Let $n\in\NN$. Then the following are equivalent:
\begin{enumerate}
\item
\label{IVt:1ia} 
$(x_n,y_n^*)=(\overline{x},\overline{y}^*)$.
\item
\label{IVt:1ib} 
$(x_n,y_n^*)\in\boldsymbol{Z}$.
\item
\label{IVt:1ic} 
$(x_n,y_n^*)\in\boldsymbol{H}_n$.
\item
\label{IVt:1id} 
$x_n=a_n$ and $y_n^*=b_n^*$.
\item
\label{IVt:1ie} 
$La_n=b_n$ and $a_n^*=-L^*b_n^*$.
\item
\label{IVt:1if} 
$\boldsymbol{H}_n=\XXX$.
\item
\label{IVt:1ig} 
$(x_{n+1/2},y_{n+1/2}^*)=(x_n,y_n^*)$.
\item
\label{IVt:1ih} 
$(x_{n+1},y_{n+1}^*)=(x_n,y_n^*)$.
\end{enumerate}
\item
\label{IVt:1ii} 
$\sum_{n\in\NN}D^f(x_{n+1},x_n)<\pinf$ and 
$\sum_{n\in\NN}D^{g^*}(y_{n+1}^*,y_n^*)<\pinf$.
\item
\label{IVt:1iii} 
$\sum_{n\in\NN}D^f(x_{n+1/2},x_n)<\pinf$ and 
$\sum_{n\in\NN}D^{g^*}(y_{n+1/2}^*,y_n^*)<\pinf$.
\item
\label{IVt:1iv}
Suppose that $f$, $g^*$, $h$, and $j$ satisfy Condition~\ref{IVC:1},
and that $\nabla h$ and $\nabla j$ are uniformly continuous on 
every bounded subset of $\intdom h$ and $\intdom j$, respectively.
Then $x_n\to\overline{x}$ and $y_n^*\to\overline{y}^*$.
\end{enumerate}
\end{theorem}
\begin{proof}
We apply Proposition~\ref{p:2} to 
\begin{equation}
\label{4e:C}
\boldsymbol{C}=\boldsymbol{Z}\cap\overline{\dom}\boldsymbol{f}.
\end{equation}
It follows from Proposition~\ref{ppOlk28ha12}\ref{ppOlk28ha12ia}
and our assumptions that $\boldsymbol{C}$ is a closed convex 
subset of $\overline{\dom}\boldsymbol{f}$ and that
$\boldsymbol{C}\cap\intdom\boldsymbol{f}\neq\emp$. 
Moreover, Proposition~\ref{pvb8uh4504}\ref{pvb8uh4504i}
asserts that $\boldsymbol{f}$ is a Legendre function.
Now let $\gamma\in[\varepsilon,\pinf[$ and let 
$\mu\in[\delta,\pinf[$.
Since $h$ is strictly convex, $\nabla h$ is 
strictly monotone \cite[Theorem~2.4.4(ii)]{IVZali02} and 
$\nabla h+\gamma A$ is likewise. Let $(x^*,x_1)$ and 
$(x^*,x_2)$ be two elements in $\gra(\nabla h+\gamma A)^{-1}$ 
such that $x_1\not=x_2$. Then $(x_1,x^*)$ and $(x_2,x^*)$ lie in 
$\gra(\nabla h+\gamma A)$ and the strict monotonicity of 
$\nabla h+\gamma A$ implies that 
\begin{equation}
0=\pair{x_1-x_2}{x^*-x^*}>0,
\end{equation}
which is impossible. Thus,
\begin{equation}
\label{4e:single1}
(\nabla h+\gamma A)^{-1}\;\text{is at most single-valued}.
\end{equation} 
The same argument shows that
\begin{equation}
\label{4e:single2}
(\nabla j+\mu B)^{-1}\;\text{is at most single-valued}.
\end{equation} 
On the other hand, by assumption, there exists
$(x,y^*)\in\boldsymbol{Z}\cap\intdom\boldsymbol{f}$.
It follows from \eqref{IVe:VKT} that $x\in\dom A$ and 
$Lx\in\dom B$. Furthermore, \eqref{ebcn5y731} yields 
$x\in\intdom f$. Therefore
\begin{equation}
\label{ehHg7yG10a}
\begin{cases}
x\in\dom A\cap\intdom f\subset\dom A\cap\intdom h\\
Lx\in\dom B\cap L(\intdom f)\subset\dom B\cap\intdom j.
\end{cases}
\end{equation}
Thus, $\dom A\cap\intdom h\neq\emp$ and
$\dom B\cap\intdom j\neq\emp$. It therefore follows from
Lemma~\ref{IVlhHg7yG02} that
\begin{equation}
\label{4e:single3}
\nabla h+\gamma A=(\nabla h+\varepsilon A)+(\gamma-\varepsilon)A 
\quad\text{and}\quad
\nabla j+\mu B=(\nabla j+\delta B)+(\mu-\delta)B
\;\;\text{are coercive}.
\end{equation}
Altogether, \eqref{4e:single1}, \eqref{4e:single2}, 
\eqref{4e:single3}, and Proposition~\ref{IVpp:20a} assert 
that the operators
\begin{equation}
\label{4:e12}
\begin{cases}
(\nabla h+\gamma A)^{-1}\colon\XX^*\to\dom A\cap\intdom h\\
(\nabla j+\mu B)^{-1}\colon\YY^*\to\dom B\cap\intdom j
\end{cases}
\end{equation}
are well defined and single-valued. Now set
\begin{equation}
\label{eKKnhf74g05}
(\forall n\in\NN)\quad
\boldsymbol{x}_n=(x_n,y_n^*)\quad\text{and}\quad
\boldsymbol{x}_{n+1/2}=(x_{n+1/2},y_{n+1/2}^*).
\end{equation}
Since \eqref{IVe:Bfg+uh4f9-09h} yields
\begin{equation}
\label{eKKnhf74g04}
(\forall n\in\NN)\quad
(a_n,a_n^*)\in\gra A\quad\text{and}\quad (b_n,b_n^*)\in\gra B,
\end{equation}
it follows from \eqref{4e:C}, 
Proposition~\ref{ppOlk28ha12}\ref{ppOlk28ha12iib},
\eqref{IVe:Bfg+uh4f9-09h}, and 
Definition~\ref{d:B-haugazeau} that
\begin{equation}
\label{eKKnhf74g03}
(\forall n\in\NN)\quad
\emp\neq\boldsymbol{C}\subset\boldsymbol{Z}\subset\boldsymbol{H}_n
=H^{\boldsymbol{f}}(\boldsymbol{x}_n,\boldsymbol{x}_{n+1/2}).
\end{equation}
Hence, appealing to 
Proposition~\ref{pvb8uh4504}\ref{pvb8uh4504i} and 
\eqref{e:2001}, we see that
\begin{equation}
\label{IVe:32a}
(\forall n\in\NN)\quad
P_{\boldsymbol{H}_n}^{\boldsymbol{f}}
\colon\intdom\boldsymbol{f}\to
\boldsymbol{H}_n\cap\intdom\boldsymbol{f}
\end{equation}
and that
\begin{equation}
\label{eKKnhf74g04i}
(\forall n\in\NN)\quad
\boldsymbol{x}_{n+1}=Q^{\boldsymbol{f}}
\big(\boldsymbol{x}_0,\boldsymbol{x}_n,\boldsymbol{x}_{n+1/2}\big).
\end{equation}
Thus, we derive from 
\eqref{eKKnhf74g05}, \eqref{eKKnhf74g03}, and 
Proposition~\ref{p:2}\ref{p:2i} that
$(x_n)_{n\in\NN}$ and $(y_n^*)_{n\in\NN}$ are well-defined
sequences in $\intdom f$ and $\intdom g^*$, respectively.

\ref{IVt:1i}: We prove the following implications.

\ref{IVt:1ia}$\Rightarrow$\ref{IVt:1ib}: Clear.

\ref{IVt:1ib}$\Rightarrow$\ref{IVt:1ia}: 
Proposition~\ref{p:2}\ref{p:2ii+}.

\ref{IVt:1ib}$\Rightarrow$\ref{IVt:1ic}: Clear by
\eqref{eKKnhf74g03}.

\ref{IVt:1ic}$\Rightarrow$\ref{IVt:1id}: In view of 
\eqref{IVe:Bfg+uh4f9-09h},
\begin{align}
0
&\geq\pair{x_n}{a_n^*+L^*b_n^*}
+\pair{b_n-La_n}{y_n^*}-\pair{a_n}{a_n^*}-\pair{b_n}{b_n^*}
\nonumber\\
&=\pair{x_n-a_n}{a_n^*+L^*y_n^*}+\pair{Lx_n-b_n}{b_n^*-y_n^*}
\nonumber\\
&=\gamma_n^{-1}\pair{x_n-a_n}{\nabla h(x_n)-\nabla h(a_n)}
+\mu_n^{-1}\pair{Lx_n-b_n}{\nabla j(Lx_n)-\nabla j(b_n)}.
\end{align}
Consequently, the strict monotonicity of $\nabla h$ and $\nabla j$ 
yields
\begin{equation}
x_n=a_n\quad\text{and}\quad Lx_n=b_n.
\end{equation}
Furthermore,
\begin{equation}
b_n^*=\mu_n^{-1}\big(\nabla j(Lx_n)-\nabla j(b_n)\big)+y_n^*
=\mu_n^{-1}\big(\nabla j(b_n)-\nabla j(b_n)\big)+y_n^*
=y_n^*.
\end{equation}

\ref{IVt:1id}$\Rightarrow$\ref{IVt:1ie}: We derive from 
\eqref{IVe:Bfg+uh4f9-09h} that
$a_n^*=\gamma_n^{-1}(\nabla h(x_n)-\nabla h(a_n))-L^*y_n^*
=-L^*y_n^*=-L^*b_n^*$. On the other hand, since
\begin{align}
\label{IWeQA8h33-30a}
\pair{La_n-b_n}{\nabla j(La_n)-\nabla j(b_n)}
&=\pair{Lx_n-b_n}{\nabla j(Lx_n)-\nabla j(b_n)}\nonumber\\
&=\mu_n\pair{Lx_n-b_n}{b_n^*-y_n^*}\nonumber\\
&=0,
\end{align}
the strict monotonicity of $\nabla j$ yields $La_n=b_n$.

\ref{IVt:1ie}$\Leftrightarrow$\ref{IVt:1if}: 
Proposition~\ref{ppOlk28ha12}\ref{ppOlk28ha12iia}.

\ref{IVt:1if}$\Rightarrow$\ref{IVt:1ig}: Indeed,
$\boldsymbol{x}_{n+1/2}=P_{\boldsymbol{H}_n}^{\boldsymbol{f}}
\boldsymbol{x}_n=\boldsymbol{x}_n$.

\ref{IVt:1ig}$\Rightarrow$\ref{IVt:1ih}: We have
\begin{equation}
\label{IWeQA8h33-30b}
\boldsymbol{x}_{n+1}=Q^{\boldsymbol{f}}(\boldsymbol{x}_0,
\boldsymbol{x}_n,\boldsymbol{x}_{n+1/2})=
P^{\boldsymbol{f}}_{H^{\boldsymbol{f}}
(\boldsymbol{x}_0,\boldsymbol{x}_n)\cap H^{\boldsymbol{f}}
(\boldsymbol{x}_n,\boldsymbol{x}_{n+1/2})}
\boldsymbol{x}_0=P^{\boldsymbol{f}}_{H^{\boldsymbol{f}}
(\boldsymbol{x}_0,\boldsymbol{x}_n)}\boldsymbol{x}_0
=\boldsymbol{x}_n.
\end{equation}

\ref{IVt:1ih}$\Rightarrow$\ref{IVt:1ig}:  By
Proposition~\ref{p:2}\ref{p:2iii-},
$0\leq D^{\boldsymbol{f}}(\boldsymbol{x}_{n+1/2},\boldsymbol{x}_n)
\leq D^{\boldsymbol{f}}(\boldsymbol{x}_{n+1},\boldsymbol{x}_n)=0$.
Therefore 
$D^{\boldsymbol{f}}(\boldsymbol{x}_{n+1/2},\boldsymbol{x}_n)=0$
and we derive from \cite[Lemma~7.3(vi)]{Ccm01}
that $\boldsymbol{x}_{n+1/2}=\boldsymbol{x}_n$.

\ref{IVt:1ig}$\Rightarrow$\ref{IVt:1ic}: Indeed,
$\boldsymbol{x}_n=\boldsymbol{x}_{n+1/2}=
P^{\boldsymbol{f}}_{\boldsymbol{H}_n}
\boldsymbol{x}_n\in\boldsymbol{H}_n$.

\ref{IVt:1id}$\Rightarrow$\ref{IVt:1ib}: 
We derive from \eqref{IVe:Bfg+uh4f9-09h} that
\begin{equation}
\label{IWeQA8h33-30c}
\nabla h(x_n)-\gamma_n L^*y_n^*\in\nabla h(a_n)+\gamma_n Aa_n
=\nabla h(x_n)+\gamma_n Ax_n. 
\end{equation}
Hence $-L^*y_n^*\in Ax_n$. Likewise, as in 
\eqref{IWeQA8h33-30a}, we first obtain $Lx_n=b_n$ and then
\begin{equation}
\label{IWeQA8h33-30d}
\nabla j(Lx_n)+\mu_n y_n^*\in\nabla j(b_n)+\mu_n Bb_n
=\nabla j(Lx_n)+\mu_n B(Lx_n).
\end{equation}
Thus, $y_n^*\in B(Lx_n)$, i.e., $Lx_n\in B^{-1}y_n^*$.
In view of \eqref{IVe:VKT}, the implication is proved.

\ref{IVt:1ii}: 
Proposition~\ref{p:2}\ref{p:2ii} yields
\begin{equation}
\label{IVe:2312b}
\sum_{n\in\NN}D^f(x_{n+1},x_n)+\sum_{n\in\NN}D^{g^*}(y_{n+1}^*,y_n^*)
=\sum_{n\in\NN}D^{\boldsymbol{f}}(\boldsymbol{x}_{n+1},
\boldsymbol{x}_n)<\pinf.
\end{equation}

\ref{IVt:1iii}: Proposition~\ref{p:2}\ref{p:2iii} yields 
\begin{equation}
\sum_{n\in\NN}D^f(x_{n+1/2},x_n)
+\sum_{n\in\NN}D^{g^*}(y_{n+1/2}^*,y_n^*)=
\sum_{n\in\NN}D^{\boldsymbol{f}}(\boldsymbol{x}_{n+1/2},
\boldsymbol{x}_n)<\pinf.
\end{equation}

\ref{IVt:1iv}: 
Proposition~\ref{p:2}\ref{p:2i} implies that
$(\boldsymbol{x}_n)_{n\in\NN}$ is a bounded sequence 
in $\intdom\boldsymbol{f}$. In turn, 
\begin{equation}
\label{eKKnhf74g15a}
(x_n)_{n\in\NN}\in(\intdom f)^{\NN}\quad\text{and}\quad
(y^*_n)_{n\in\NN}\in(\intdom g^*)^{\NN}\quad\text{are bounded}.
\end{equation}
On the other hand, by \eqref{IVe:Bfg+uh4f9-09h},
\begin{equation}
\label{IWeQA8h33-31a}
(\forall n\in\NN)\quad
(x_{n+1/2},y_{n+1/2}^*)=
\boldsymbol{x}_{n+1/2}=P_{\boldsymbol{H_n}}^{\boldsymbol{f}}
\boldsymbol{x}_n\in\boldsymbol{H}_n
\end{equation}
and
\begin{equation}
\label{IVe:IV8}
(\forall n\in\NN)\quad\pair{x_{n+1/2}}{a_n^*+L^*b_n^*}
+\pair{b_n-La_n}{y_{n+1/2}^*}
=\pair{a_n}{a_n^*}+\pair{b_n}{b_n^*}.
\end{equation}
Therefore,
\begin{align}
\label{IVe:8re84}
(\forall n\in\NN)\quad
&\|x_n-x_{n+1/2}\|\,\|a_n^*+L^*b_n^*\|
+\|b_n-La_n\|\,\|y_n^*-y_{n+1/2}^*\|\nonumber\\
&\geq\pair{x_n-x_{n+1/2}}{a_n^*+L^*b_n^*}
+\pair{b_n-La_n}{y_n^*-y_{n+1/2}^*}\nonumber\\
&=\pair{x_n}{a_n^*+L^*b_n^*}
+\pair{b_n-La_n}{y_n^*}
-\pair{a_n}{a_n^*}-\pair{b_n}{b_n^*}\nonumber\\
&=\pair{x_n-a_n}{a_n^*+L^*y_n^*}
+\pair{Lx_n-b_n}{b_n^*-y_n^*}\nonumber\\
&=\gamma_n^{-1}\pair{x_n-a_n}{\nabla h(x_n)
-\nabla h(a_n)}+\mu_n^{-1}
\pair{Lx_n-b_n}{\nabla j(Lx_n)-\nabla j(b_n)}
\nonumber\\
&\geq\sigma^{-1}\big(D^h(x_n,a_n)+D^h(a_n,x_n)
+D^j(Lx_n,b_n)+D^j(b_n,Lx_n)\big)\nonumber\\
&\geq\sigma^{-1}\big(D^h(x_n,a_n)+D^j(Lx_n,b_n)\big).
\end{align}
However, since \ref{IVt:1iii} yields
\begin{equation}
\label{IVe:IV4}
D^f(x_{n+1/2},x_n)\to 0
\quad\text{and}\quad
D^{g^*}(y_{n+1/2}^*,y_n^*)\to 0
\end{equation}
and since $f$ and $g^*$ satisfy Condition~\ref{IVC:1}, 
\eqref{IVbcn5y713b} yields
\begin{equation}
\label{IVe:IV5}
x_{n+1/2}-x_n\to 0
\quad\text{and}\quad
y_{n+1/2}^*-y_n^*\to 0.
\end{equation}
Since $\nabla h$ is uniformly continuous on every bounded subset of 
$\intdom h$, Lemma~\ref{IVlKKnhf74g10} asserts that $\nabla h$ 
is bounded on every bounded subset of $\intdom h$ and hence, 
since $\inte\dom f\subset\inte\dom h$ and $L^*$ is bounded, 
it follows from \eqref{eKKnhf74g15a} that
$\big(\nabla h(x_n)-\gamma_nL^*y_n^*\big)_{n\in\NN}$ is bounded. 
We therefore deduce from \eqref{4:e12}, 
\eqref{IVe:Bfg+uh4f9-09h},
and Lemma~\ref{IVlKKnhf74g12} that
\begin{equation}
\label{4:e18}
(a_n)_{n\in\NN}\in(\intdom h)^{\NN}\;\;\text{is bounded}.
\end{equation}
Similarly, since $\nabla j$ is uniformly continuous on every 
bounded subset 
of $\intdom j$ and $L(\intdom f)\subset\intdom j$, it follows 
from \eqref{eKKnhf74g15a} and Lemma~\ref{IVlKKnhf74g10} 
that $\big(\nabla j(Lx_n)+\mu_ny_n^*\big)_{n\in\NN}$ is bounded 
and hence \eqref{4:e12}, \eqref{IVe:Bfg+uh4f9-09h},
and Lemma~\ref{IVlKKnhf74g12} yield
\begin{equation}
\label{4:e19}
(b_n)_{n\in\NN}\in(\intdom j)^{\NN}\;\;\text{is bounded}.
\end{equation} 
Thus, $(\nabla h(x_n))_{n\in\NN}$, $(\nabla h(a_n))_{n\in\NN}$,
$(\nabla j(Lx_n))_{n\in\NN}$, and $(\nabla j(b_n))_{n\in\NN}$ 
are bounded and we deduce from \eqref{IVe:Bfg+uh4f9-09h} that 
\begin{equation}
\label{eKKnhf74g15b}
(a^*_n)_{n\in\NN}\quad\text{and}\quad(b^*_n)_{n\in\NN}
\quad\text{are bounded}.
\end{equation}
We therefore derive from 
\eqref{IVe:8re84}, \eqref{IVe:IV5}, \eqref{4:e18}, and 
\eqref{4:e19} that
\begin{equation}
\label{eKKnhf74g15h}
D^h(x_n,a_n)\to 0\quad\text{and}\quad
D^j(Lx_n,b_n)\to 0.
\end{equation}
Since $h$ and $j$ satisfy Condition~\ref{IVC:1}, we get
\begin{equation}
\label{IVe:IV10}
x_n-a_n\to 0\quad\text{and}\quad Lx_n-b_n\to 0.
\end{equation}
Therefore, since $\nabla h$ is uniformly continuous on every 
bounded subset of $\intdom h$ and $\nabla j$ is uniformly 
continuous on every bounded subset of $\intdom j$,
\begin{equation}
\label{IVe:IV12a}
\nabla h(x_n)-\nabla h(a_n)\to 0\quad\text{and}\quad
\nabla j(Lx_n)-\nabla j(b_n)\to 0.
\end{equation}
Hence, using \eqref{IVe:Bfg+uh4f9-09h}, we get
\begin{equation}
\label{IVe:IV13}
a_n^*+L^*y_n^*\to 0
\quad\text{and}\quad 
b_n^*-y_n^*\to 0.
\end{equation}
Now, let 
$\boldsymbol{x}=(x,y^*)\in\WC(\boldsymbol{x}_n)_{n\in\NN}$,
say $\boldsymbol{x}_{k_n}\weakly\boldsymbol{x}$. 
Then $x_{k_n}\weakly x$ and $y_{k_n}^*\weakly y^*$, and we 
derive from \eqref{IVe:IV10} and \eqref{IVe:IV13} that
\begin{equation}
\label{IVe:IV11}
\begin{cases}
a_{k_n}\weakly x\\
b^*_{k_n}\weakly y^*
\end{cases}
\quad\text{and}\quad
\begin{cases}
La_{k_n}-b_{k_n}\to 0\\
a^*_{k_n}+L^*b^*_{k_n}\to 0.
\end{cases}
\end{equation}
It therefore follows from \eqref{eKKnhf74g04},
Proposition~\ref{ppOlk28ha12}\ref{ppOlk28ha12iii},
and \eqref{eKKnhf74g15a} that
$\boldsymbol{x}\in\boldsymbol{Z}\cap\overline{\dom}\boldsymbol{f}
=\boldsymbol{C}$. Hence, we derive from 
Proposition~\ref{p:2}\ref{p:2iv} that 
\begin{equation}
\label{eKKnhf74g15x}
D^{f}({x}_n,\overline{x})+D^{g^*}(y^*_n,\overline{y}^*)=
D^{\boldsymbol{f}}(\boldsymbol{x}_n,\overline{\boldsymbol{x}})\to 0,
\end{equation}
where $\overline{\boldsymbol{x}}=(\overline{x},\overline{y}^*)$.
Hence, $D^{f}({x}_n,\overline{x})\to 0$, 
$D^{g^*}(y^*_n,\overline{y}^*)\to 0$, and, 
since $f$ and $g^*$ satisfy Condition~\ref{IVC:1}, we conclude
that $x_n\to\overline{x}$ and $y^*_n\to\overline{y}^*$.
\end{proof}

\begin{remark}
\label{IVrKKnhf74g31}
We provide a couple of settings that satisfy the assumptions of
Theorem~\ref{IVt:1}.
\begin{enumerate}
\item
In Problem~\ref{prob:1}, suppose that $\XX$ and $\YY$ are Hilbert 
spaces, that $f=\|\cdot\|^2/2$, and that $g=\|\cdot\|^2/2$.
Furthermore, in Theorem~\ref{IVt:1}, set $h=f$ and $j=g$, and note
that, for any $\varepsilon\in\RPP$, 
$\nabla h+\varepsilon A=\Id+\varepsilon A$ and 
$\nabla j+\varepsilon B=\Id+\varepsilon B$ are strongly monotone 
and hence coercive by Lemma~\ref{l:20}\ref{l:20ii}.
Then we recover the framework of \cite{Nfao15}, 
which has been applied to domain decomposition problems in 
\cite{Atto15}. 
\item
Let $(\Omega_1,\mathcal{F}_1,\mu_1)$ and
$(\Omega_2,\mathcal{F}_2,\mu_2)$ be measure spaces,
let $p$ and $q$ be in $\left]1,\pinf\right[$, and set
$p^*=p/(p-1)$ and $q^*=q/(q-1)$. In Problem~\ref{prob:1}, suppose 
that $\XX=L^p(\Omega_1,\mathcal{F}_1,\mu_1)$, 
$\YY=L^q(\Omega_2,\mathcal{F}_2,\mu_2)$, $f=\|\cdot\|^p/p$, 
and $g=\|\cdot\|^q/q$. Then
$\XX^*=L^{p^*}(\Omega_1,\mathcal{F}_1,\mu_1)$,
$\YY^*=L^{q^*}(\Omega_2,\mathcal{F}_2,\mu_2)$, and
$g^*=\|\cdot\|^{q^*}/q^*$. Moreover, 
it follows from Clarkson's theorem 
\cite[Theorem~II.4.7]{IVCi90_B} that $\XX$, $\XX^*$,
$\YY$, and $\YY^*$ are uniformly convex and uniformly smooth.
Hence, we derive from \cite[Corollary~5.5 and Example~6.5]{Ccm01} 
that $f$, $g$, and $g^*$ are Legendre functions which are 
uniformly convex on every bounded set, and which therefore 
satisfy Condition~\ref{IVC:1} by virtue of 
\cite[Example~4.10(i)]{Sico03}. 
Now set $h=f$ and $j=g$ in Theorem~\ref{IVt:1}.
We derive from 
\cite[Theorem~II.2.16(i)]{IVCi90_B} that $\nabla h$ and $\nabla j$ 
are uniformly continuous on every bounded subset of $\XX$ and 
$\YY$, respectively. 
In addition, $h$ and $j$ are supercoercive and therefore,
for any $\varepsilon\in\RPP$, 
Proposition~\ref{IVpp:20a}\ref{IVcl:20b4} asserts that 
$\nabla h+\varepsilon A$ and $\nabla j+\varepsilon B$ are 
coercive. Finally, it follows from 
\cite[Proposition~II.4.9]{IVCi90_B} that
$\nabla h\colon x\mapsto |x|^{p-1}\sign(x)$ and 
$\nabla j\colon y\mapsto |y|^{q-1}\sign(y)$.
\end{enumerate}
\end{remark}

\begin{remark}
The implementation of algorithm \eqref{IVe:Bfg+uh4f9-09h}
requires the evaluation of the operator $(\nabla h+A)^{-1}$.
We provide a simple example in the Euclidean plane $\XX$ of a
maximally monotone operator $A$ for which 
$(\nabla h+A)^{-1}$ can be computed explicitly, whereas
the classical resolvent $(\Id+A)^{-1}$ is difficult to evaluate.
Let $\beta\in\RPP$ and let $\psi\colon\RR\to\RR$ be a Legendre 
function with a $\beta$ Lipschitz-continuous derivative.
Set
\begin{equation}
A\colon\RR^2\to\RR^2\colon (\xi_1,\xi_2)\mapsto
\big(\beta\xi_1-\psi'(\xi_1)-\xi_2,
\xi_1+\beta\xi_2-\psi'(\xi_2)\big)
\end{equation}
and 
\begin{equation}
h\colon\RR^2\to\RX\colon (\xi_1,\xi_2)\mapsto
\psi(\xi_1)+\psi(\xi_2).
\end{equation}
Then it follows from \cite[Theorem~18.15]{Livre1} that $A$ is the
sum of the gradient of the convex function 
$(\xi_1,\xi_2)\mapsto\beta\xi_1^2/2-\psi(\xi_1)+
\beta\xi_2^2/2-\psi(\xi_2)$ 
and of the skew linear operator $(\xi_1,\xi_2)\mapsto(-\xi_2,\xi_1)$.
Thus, $A$ is a maximally monotone operator 
\cite[Corollary~24.4]{Livre1} which is not the
subdifferential of a convex function. In addition, as in 
Proposition~\ref{pvb8uh4504}\ref{pvb8uh4504i},
$h$ is a Legendre function and
\begin{equation}
\big(\forall (\xi_1,\xi_2)\in\RR^2\big)\quad 
(\nabla h+A)^{-1}(\xi_1,\xi_2)
=\Bigg(\dfrac{\beta\xi_1+\xi_2}{1+\beta^2},
\dfrac{\beta\xi_2-\xi_1}{1+\beta^2}\Bigg).
\end{equation}
\end{remark}

\begin{remark}
\label{rvb8uh4504}
At every iteration $n$, algorithm \eqref{IVe:Bfg+uh4f9-09h}
requires the computation of $\boldsymbol{x}_{n+1/2}=
P_{\boldsymbol{H}_n}^{\boldsymbol{f}}(x_n,y_n^*)$ and then of
$\boldsymbol{x}_{n+1}=Q^{\boldsymbol{f}}((x_0,y_0^*),(x_n,y_n^*),
(x_{n+1/2},y_{n+1/2}^*))$. 
Set $\boldsymbol{s}_n^*=(a_n^*+L^*b_n^*,b_n-La_n)$, 
$\eta_n=\pair{a_n}{a_n^*}+\pair{b_n}{b_n^*}$, and
$\boldsymbol{x}_n=(x_n,y_n^*)$,
Then, if $\boldsymbol{x}_n\not\in\boldsymbol{H}_n$, 
$\boldsymbol{x}_{n+1/2}$ is the Bregman projection of 
$\boldsymbol{x}_n$ onto the closed affine 
hyperplane $\menge{\boldsymbol{x}\in\XXX}
{\pair{\boldsymbol{x}}{\boldsymbol{s}_n^*}=\eta_n}$. Thus,
$\boldsymbol{x}_{n+1/2}$ is the solution of the problem 
\begin{equation}
\label{evb8uh4505a}
\minimize{\pair{\boldsymbol{p}}{\boldsymbol{s}_n^*}=\eta_n}
{\boldsymbol{f}(\boldsymbol{p})-\pair{\boldsymbol{p}}{\nabla
\boldsymbol{f}(\boldsymbol{x}_n)}}
\end{equation}
which, using standard first order conditions, is characterized by
(see also \cite[Remark~6.13]{BB97} and 
\cite[Lemma~2.2.1]{Cens97}) 
\begin{equation}
\label{evb8uh4505b}
\begin{cases}
\nabla\boldsymbol{f}(\boldsymbol{x}_{n+1/2})=\nabla
\boldsymbol{f}(\boldsymbol{x}_n)-\lambda\boldsymbol{s}_n^*\\
\pair{\boldsymbol{x}_{n+1/2}}{\boldsymbol{s}_n^*}=\eta_n\\
\lambda\in\RPP.
\end{cases}
\end{equation}
In view of \cite[Theorem~5.10]{Ccm01}, the Lagrange multiplier 
$\lambda$ is uniquely determined by the equation 
$\pair{\nabla\boldsymbol{f}^*(\nabla\boldsymbol{f}
(\boldsymbol{x}_n)-\lambda\boldsymbol{s}_n^*)}
{\boldsymbol{s}_n^*}=\eta_n$. The problem therefore reduces to
finding the solution $\overline{\lambda}$ to this equation in $\RPP$
and then setting $\boldsymbol{x}_{n+1/2}=\nabla\boldsymbol{f}^*(
\nabla\boldsymbol{f}(\boldsymbol{x}_n)-\overline{\lambda}
\boldsymbol{s}_n^*)$. Likewise, it follows from \eqref{e:Q} that 
$\boldsymbol{x}_{n+1}$ is the unique solution to the problem 
\begin{equation}
\label{evb8uh4505c}
\minimize{\substack{\pair{\boldsymbol{p}-\boldsymbol{x}_n}
{\nabla\boldsymbol{f}(\boldsymbol{x}_0)-
\nabla\boldsymbol{f}(\boldsymbol{x}_n)}\leq 0\\
\pair{\boldsymbol{p}-\boldsymbol{x}_{n+1/2}}
{\nabla\boldsymbol{f}(\boldsymbol{x}_{n})-
\nabla\boldsymbol{f}(\boldsymbol{x}_{n+1/2})}\leq 0}}
{\boldsymbol{f}(\boldsymbol{p})-\pair{\boldsymbol{p}}{\nabla
\boldsymbol{f}(\boldsymbol{x}_0)}}.
\end{equation}
Depending on the number of active constraints, this 
problem boils down to determining 
up to two Lagrange multipliers in $\RPP$.
\end{remark}

Next, we consider a specialization of Problem~\ref{prob:1} to
multivariate structured minimization.

\begin{problem}
\label{IVpb:V3}
Let $m$ and $p$ be strictly positive integers, let
$(\XX_i)_{1\leq i\leq m}$ and $(\YY_k)_{1\leq k\leq p}$ be 
reflexive real Banach spaces, and let 
$\boldsymbol{\XX}$ be the standard vector product space 
$\Big(\Cart_{\!\!i=1}^{\!\!m}\XX_i\Big)\times
\Big(\Cart_{\!\!k=1}^{\!\!p}\YY_k^*\Big)$ 
equipped with the norm 
\begin{equation}
(x,y^*)=\big((x_i)_{1\leq i\leq m},(y_k^*)_{1\leq k\leq p}\big)
\mapsto\sqrt{\sum_{i=1}^m\|x_i\|^2+\sum_{k=1}^p\|y_k^*\|^2}.
\end{equation} 
For every $i\in\{1,\ldots,m\}$ and every $k\in\{1,\ldots,p\}$, 
let $\varphi_i\in\Gamma_0(\XX_i)$, let $\psi_k\in\Gamma_0(\YY_k)$,
and let $L_{ki}\colon\XX_i\to\YY_k$ be linear and bounded. 
Consider the primal problem 
\begin{equation}
\label{4eKKnhf74g201}
\minimize{x_1\in\XX_1,\ldots,\,x_m\in\XX_m}{\sum_{i=1}^m
\varphi_i(x_i)+\sum_{k=1}^p 
\psi_k\bigg(\sum_{i=1}^mL_{ki}x_i\bigg)},
\end{equation}
the dual problem
\begin{equation}
\label{4eKKnhf74g202}
\minimize{y^*_1\in\YY_1^*,\ldots,\,y^*_p\in\YY_p^*}{\sum_{i=1}^m
\varphi_i^*\bigg(-\sum_{k=1}^pL_{ki}^*y^*_k\bigg)
+\sum_{k=1}^p\psi^*_k(y^*_k)}, 
\end{equation}
and let 
\begin{multline}
\boldsymbol{Z}=\bigg\{({x},{y}^*)
\in\boldsymbol{\mathcal{X}}\;\bigg |\;
(\forall i\in\{1,\ldots,m\})\;\;-\sum_{k=1}^pL_{ki}^*y_k^*
\in\partial\varphi_i(x_i)\:\;\text{and}\\
(\forall k\in\{1,\ldots,p\})\;\;\sum_{i=1}^mL_{ki}x_i\in
\partial\psi_k^*(y_k^*)\bigg\}
\end{multline}
be the associated Kuhn-Tucker set. For every $i\in\{1,\ldots,m\}$,
let $f_i\in\Gamma_0(\XX_i)$ be a Legendre function and let
$x_{i,0}\in\inte\dom f_i$. For every $k\in\{1,\ldots,p\}$, 
let $g_k\in\Gamma_0(\YY_k)$ be a Legendre function and let
$y^*_{k,0}\in\inte\dom g^*_k$. Set
${x}_0=(x_{i,0})_{1\leq i\leq m}$, 
${y}_0^*=(y_{k,0}^*)_{1\leq k\leq p}$, and 
\begin{equation}
\boldsymbol{f}\colon\boldsymbol{\XX}\to\RX\colon
({x},{y}^*)\mapsto
\sum_{i=1}^mf_i(x_i)+\sum_{k=1}^pg_k^*(y_k^*),
\end{equation}
and suppose that 
$\boldsymbol{Z}\cap\inte\dom\boldsymbol{f}\neq\emp$. 
The objective is to find the best Bregman approximation 
$\big((\overline{x}_i)_{1\leq i\leq m},
(\overline{y}_k^*)_{1\leq k\leq p}\big)
=P_{\boldsymbol{Z}}^{\boldsymbol{f}} ({x}_0,{y}_0^*)$ 
to $({x}_0,{y}_0^*)$ from $\boldsymbol{Z}$.
\end{problem}

We derive from Theorem~\ref{IVt:1} the following convergence
result for a splitting algorithm to solve Problem~\ref{IVpb:V3}.

\begin{proposition}
\label{4p:1}
Consider the setting of Problem~\ref{IVpb:V3}. For every 
$i\in\{1,\ldots,m\}$, let $h_i\in\Gamma_0(\XX_i)$ be a Legendre 
function such that $\intdom f_i\subset\intdom h_i$ and 
$h_i+\varepsilon_i\varphi_i$ is supercoercive for some
$\varepsilon_i\in\RPP$.
For every $k\in\{1,\ldots,p\}$, let $j_k\in\Gamma_0(\YY_k)$ 
be a Legendre function such that 
$\sum_{i=1}^mL_{ki}(\intdom f_i)\subset\intdom j_k$ and 
$j_k+\delta_k\psi_k$ is supercoercive for some $\delta_k\in\RPP$. 
Set $\varepsilon=\max_{1\leq i\leq m}\varepsilon_i$ and
$\delta=\max_{1\leq k\leq p}\delta_k$, let 
$\sigma\in\left[\max\{\varepsilon,\delta\},+\infty\right[$, and 
iterate
\begin{equation}
\label{4eKKnhf74g203}
\begin{array}{l}
\text{for}\;n=0,1,\ldots\\
\left\lfloor
\begin{array}{l}
(\gamma_n,\mu_n)\in\left[\varepsilon,\sigma\right]
\times\left[\delta,\sigma\right]\\[1mm]
\text{for}\;i=1,\ldots,m\\
\left\lfloor
\begin{array}{l}
a_{i,n}=(\nabla h_i+\gamma_n\partial\varphi_i)^{-1}
\Big(\nabla h_i(x_{i,n})
-\gamma_n\sum_{k=1}^pL_{ki}^*y_{k,n}^*\Big)\\[2mm]
a_{i,n}^*=\gamma_n^{-1}\big(\nabla h_i(x_{i,n})
-\nabla h_i(a_{i,n})\big)-\sum_{k=1}^pL_{ki}^*y_{k,n}^*\\[1mm]
\end{array}
\right.\\[3mm]
\text{for}\;k=1,\ldots,p\\
\left\lfloor
\begin{array}{l}
b_{k,n}=(\nabla j_k+\mu_n\partial\psi_k)^{-1}
\big(\nabla j_k\big(\sum_{i=1}^mL_{ki}x_{i,n}\big)
+\mu_ny_{k,n}^*\big)\\[1mm]
b_{k,n}^*=\mu_n^{-1}\big(\nabla j_k\big(\sum_{i=1}^m
L_{ki}x_{i,n}\big)-\nabla j_k(b_{k,n})\big)+y_{k,n}^*\\[1mm]
t_{k,n}=b_{k,n}-\sum_{i=1}^mL_{ki}a_{i,n}\\[1mm]
\end{array}
\right.\\
\text{for}\;i=1,\ldots,m\\
\left\lfloor
\begin{array}{l}
s_{i,n}^*=a_{i,n}^*+\sum_{k=1}^pL_{ki}^*b_{k,n}^*\\[1mm]
\end{array}
\right.\\
\eta_n=\sum_{i=1}^m\pair{a_{i,n}}{a_{i,n}^*}
+\sum_{k=1}^p\pair{b_{k,n}}{b_{k,n}^*}\\[1mm]
\boldsymbol{H}_n=
\Menge{({x},{y}^*)\in\XXX}
{\sum_{i=1}^m\pair{x_i}{s_{i,n}^*}
+\sum_{k=1}^p\pair{t_{k,n}}{y_k^*}
\leq\eta_n}\\[1mm]
\big({x}_{n+1/2},{y}_{n+1/2}^*\big)
=P^{\boldsymbol{f}}_{\boldsymbol{H}_n}
({x}_n,{y}^*_n)\\[1mm]
\big({x}_{n+1},{y}_{n+1}^*\big)=
Q^{{f}}\big(({x}_0,{y}_0^*),
({x}_n,{y}^*_n),
({x}_{n+1/2},{y}_{n+1/2}^*)\big),\\[1mm]
\end{array}
\right.\\
\end{array}
\end{equation}
where we use the notation $(\forall n\in\NN)$ 
$x_n=(x_{i,n})_{1\leq i\leq m}$ and
$y^*_n=(y^*_{k,n})_{1\leq k\leq p}$. 
Suppose that the following hold: 
\begin{enumerate}
\item
For every $i\in\{1,\ldots,m\}$, $f_i$ and $h_i$ satisfy 
Condition~\ref{IVC:1} and $\nabla h_i$ is uniformly continuous on 
every bounded subset of $\inte\dom h_i$.
\item
For every $k\in\{1,\ldots,p\}$, $g_k^*$ and $j_k$ satisfy 
Condition~\ref{IVC:1} and $\nabla j_k$ is uniformly continuous on 
every bounded subset of $\inte\dom j_k$.
\end{enumerate}
Then 
\begin{equation}
(\forall i\in\{1,\ldots,m\})\quad 
x_{i,n}\to\overline{x}_i
\quad\text{and}\quad 
(\forall k\in\{1,\ldots,p\})\quad 
y_{k,n}^*\to\overline{y}_k^*.
\end{equation}
\end{proposition}
\begin{proof}
Denote by $\XX$ and $\YY$ the standard vector product spaces 
$\Cart_{\!\!i=1}^{\!\!m}\XX_i$ and 
$\Cart_{\!\!k=1}^{\!\!p}\YY_k$ equipped with the norms 
$x=(x_i)_{1\leq i\leq m}\mapsto\sqrt{\sum_{i=1}^m\|x_i\|^2}$ and 
$y=(y_k)_{1\leq k\leq p}\mapsto\sqrt{\sum_{k=1}^p\|y_k\|^2}$,
respectively. Then $\XX^*$ is the vector product space 
$\Cart_{\!\!i=1}^{\!\!m}\XX_i^*$ equipped with the norm 
$x^*\mapsto\sqrt{\sum_{i=1}^m\|x_i^*\|^2}$ 
and $\YY^*$ is the vector product space 
$\Cart_{\!\!k=1}^{\!\!p}\YY_k^*$ equipped with the norm 
$y^*\mapsto\sqrt{\sum_{k=1}^p\|y_k^*\|^2}$. 
Let us introduce the operators
\begin{equation}
\label{4e:opt1}
\begin{cases}
A\colon\XX\to 2^{\XX^*}\colon x\mapsto
\Cart_{\!\!i=1}^{\!\!m}\partial\varphi_i(x_i)\\
B\colon\YY\to 2^{\YY^*}\colon y
\mapsto\Cart_{\!\!k=1}^{\!\!p}\partial\psi_k(y_k)\\
L\colon\XX\to\YY\colon x\mapsto
\big(\sum_{i=1}^mL_{ki}x_i\big)_{1\leq k\leq p}
\end{cases}
\end{equation}
and the functions
\begin{equation}
\label{4e:funct1}
\begin{cases}
f\colon\XX\to\RX\colon x\mapsto\sum_{i=1}^mf_i(x_i)\\
h\colon\XX\to\RX\colon x\mapsto\sum_{i=1}^mh_i(x_i)\\
\varphi\colon\XX\to\RX\colon x\mapsto\sum_{i=1}^m\varphi_i(x_i)\\
g\colon\YY\to\RX\colon y\mapsto\sum_{k=1}^pg_k(y_k)\\
j\colon\YY\to\RX\colon y\mapsto\sum_{k=1}^pj_k(y_k).
\end{cases}
\end{equation}
Then it follows from \cite[Theorem~3.1.11]{IVZali02} that $A$ 
and $B$ are maximally monotone. In addition, the adjoint of $L$ is 
$L^*\colon\YY^*\to\XX^*\colon y^*\mapsto(\sum_{k=1}^p
L_{ki}^*y_k^*)_{1\leq i\leq m}$, and, as in
Proposition~\ref{pvb8uh4504}\ref{pvb8uh4504i},
$f$ and $g$ are Legendre functions. 
Thus, Problem~\ref{IVpb:V3} is a special case of 
Problem~\ref{prob:1}. Furthermore, $h$ and $j$ are Legendre
functions,
\begin{equation}
\label{ekIuh711a}
\intdom f=\underset{i=1}{\overset{m}{\Cart}}\inte\dom f_i\subset
\underset{i=1}{\overset{m}{\Cart}}\inte\dom h_i=\intdom h,
\end{equation}
and 
\begin{equation}
\label{ekIuh711b}
L\big(\intdom f\big)=\underset{k=1}{\overset{p}{\Cart}}
\sum_{i=1}^m L_{ki}(\intdom f_i)\subset
\underset{k=1}{\overset{p}{\Cart}}\inte\dom j_k=\intdom j.
\end{equation}
Next we observe that, for every $i\in\{1,\ldots,m\}$, since 
$h_i+\varepsilon\varphi_i$ is supercoercive, 
$(h_i+\varepsilon\varphi_i)^*$ is bounded above on
every bounded subset of $\XX_i^*$ \cite[Theorem~3.3]{Ccm01}.
As a result, 
$(h+\varepsilon\varphi)^*\colon x^*\mapsto\sum_{i=1}^m
(h_i+\varepsilon\varphi_i)^*(x_i^*)$ 
is bounded above on every bounded subset of $\XX^*$, and it follows
from \cite[Theorem~3.3]{Ccm01} that $h+\varepsilon\varphi$ is
supercoercive. In turn since, as in \eqref{ehHg7yG10a},
$\emp\neq\dom A\cap\intdom f\subset
\dom\varphi\cap\intdom f$, we derive from 
\cite[Theorem~2.8.3]{IVZali02}
and Lemma~\ref{l:20}\ref{l:20iii} that
\begin{equation}
\label{ekIuh713}
\nabla h+\varepsilon A=\nabla h+\varepsilon\partial\varphi=
\partial(h+\varepsilon\varphi) 
\end{equation}
is coercive. We show in a similar fashion that $\nabla j+\delta B$ 
is coercive. Now set, for every $n\in\NN$, 
$a_n=(a_{i,n})_{1\leq i\leq m}$, 
$a_n^*=(a_{i,n}^*)_{1\leq i\leq m}$, 
$b_n=(b_{k,n})_{1\leq k\leq p}$, and
$b_n^*=(b_{k,n}^*)_{1\leq k\leq p}$. Then, for every $n\in\NN$,
we have
\begin{align}
&\hskip -5mm (\forall i\in\{1,\ldots,m\})\quad
a_{i,n}=(\nabla h_i+\gamma_n\partial\varphi_i)^{-1}\bigg(\nabla
h_i(x_{i,n})-\gamma_n\sum_{k=1}^pL_{ki}^*y_{k,n}^*\bigg)\nonumber\\
&\Leftrightarrow (\forall i\in\{1,\ldots,m\})\quad 
\nabla h_i(x_{i,n})-\gamma_n\sum_{k=1}^pL_{ki}^*y_{k,n}^*
\in\nabla h_i(a_{i,n})+\gamma_n\partial\varphi_i(a_{i,n})
\nonumber\\
&\Leftrightarrow\nabla h(x_n)-\gamma_nL^*y_n^*
\in\nabla h(a_n)+\gamma_nAa_n\nonumber\\
&\Leftrightarrow a_n=(\nabla h+\gamma_nA)^{-1}\big(\nabla h(x_n)
-\gamma_nL^*y_n^*\big).
\end{align}
Likewise,
\begin{equation}
(\forall n\in\NN)\quad 
b_n=(\nabla j+\mu_nB)^{-1}\big(\nabla j(Lx_n)+\mu_ny_n^*\big).
\end{equation}
Thus, \eqref{4eKKnhf74g203} is a special case of 
\eqref{IVe:Bfg+uh4f9-09h}. In addition, it follows from our
assumptions and \eqref{4e:funct1} that $f$, $g^*$, $h$, and $j$ 
satisfy Condition~\ref{IVC:1}, and that $\nabla h$ and $\nabla j$ 
are uniformly continuous on every bounded subset of $\intdom h$ 
and $\intdom j$, respectively. Altogether, the conclusions follow
from Theorem~\ref{IVt:1}\ref{IVt:1iv}, with 
$\overline{x}=(\overline{x}_i)_{1\leq i\leq m}$ and
$\overline{y}^*=(\overline{y}_k^*)_{1\leq k\leq p}$.
\end{proof}

\begin{remark}
In Problem~\ref{IVpb:V3}, suppose that, for every 
$i\in\{1,\ldots,m\}$ and every $k\in\{1,\ldots,p\}$, 
$\varphi_i$ and $\psi_k$ are supercoercive 
Legendre functions satisfying Condition~\ref{IVC:1}, 
that $\nabla\varphi_i$ and $\nabla\psi_k$ are uniformly 
continuous on bounded subset of $\inte\dom\varphi_i$ 
and $\inte\dom\psi_k$, respectively, and that
$\sum_{i=1}^mL_{ki}(\inte\dom\varphi_i)\subset\inte\dom\psi_k$. 
Then, in Proposition~\ref{4p:1}, we can choose, for every 
$i\in\{1,\ldots,m\}$ and every $k\in\{1,\ldots,p\}$, 
$h_i=\varphi_i$ and $j_k=\psi_k$, and in \eqref{4eKKnhf74g203}, 
we obtain
\begin{equation}
\label{4ekIuh729a}
a_{i,n}=\nabla h_i^*\left(\frac{\nabla h_i(x_{i,n})
-\gamma_n\Sum_{k=1}^pL_{ki}^*y_{k,n}^*}{1+\gamma_n}\right)
\end{equation}
and
\begin{equation}
b_{k,n}=\nabla j_k^*\left(\frac{\nabla j_k\bigg(\Sum_{i=1}^m
L_{ki}x_{i,n}\bigg)+\mu_ny_{k,n}^*}{1+\mu_n}\right).
\end{equation}
For example, suppose that, for every $i\in\{1,\ldots,m\}$ 
and every $k\in\{1,\ldots,p\}$, $\XX_i=\RR$, 
$\YY_k=\RR$, and $\varphi_i=h_i$ is the Hellinger-like 
function, i.e.,
\begin{equation}
\varphi_i\colon\RR\to\RX\colon x_i\mapsto
\begin{cases}
-\sqrt{1-x_i^2},&\text{if}\;\;
x_i\in\left[-1,1\right];\\
\pinf,&\text{otherwise}.
\end{cases}
\end{equation}
Then \eqref{4ekIuh729a} becomes
\begin{equation}
\label{4ekIuh729b}
a_{i,n}=\frac{x_{i,n}-\gamma_n
\bigg(\Sum_{k=1}^pL_{ki}^*y_{k,n}^*\bigg)\sqrt{1-x_{i,n}^2}}
{\sqrt{(1+\gamma_n)^2(1-x_{i,n}^2)+\bigg(x_{i,n}
-\gamma_n\bigg(\Sum_{k=1}^pL_{ki}^*y_{k,n}^*\bigg)
\sqrt{1-x_{i,n}^2}\bigg)^2}}.
\end{equation}
Furthermore, as shown in the next section, in 
finite-dimensional spaces, we can remove Condition~\ref{IVC:1} 
and the assumption on the uniform continuity of 
$(\nabla\varphi_i)_{1\leq i\leq m}$ and 
$(\nabla\psi_k)_{1\leq k\leq p}$.
\end{remark}

\section{Finite-dimensional setting}
\label{sec:4}

In finite-dimensional spaces, the convergence of 
algorithm~\eqref{IVe:Bfg+uh4f9-09h} can be obtained 
under more general assumptions. To establish the corresponding
results, the following technical facts will be needed.

\begin{lemma}
\label{lkIuh727}
Let $\XX$ be a finite-dimensional real Banach space and let
$f\in\Gamma_0(\XX)$ be a Legendre function. Then the following
hold:
\begin{enumerate}
\item
\label{lkIuh727i--}
$f$ and $\nabla f$ are continuous on $\intdom f$
{\rm \cite[Corollaries~8.30(iii),~17.34, and~17.35]{Livre1}}.
\item
\label{lkIuh727i-}
$\nabla f\colon\intdom f\to\intdom f^*$ is bijective with inverse
$\nabla f^*\colon\intdom f^*\to\intdom f$
{\rm \cite[Theorem~5.10]{Ccm01}}.
\item
\label{lkIuh727i}
Let $x\in\intdom f$, let $y\in\overline{\dom}f$, and let 
$(y_n)_{n\in\NN}\in(\intdom f)^{\NN}$. Suppose that $y_n\to y$
and that $(D^f(x,y_n))_{n\in\NN}$ is bounded. Then $y\in\intdom f$
and $D^f(y,y_n)\to 0$ {\rm \cite[Theorem~3.8(ii)]{BB97}}.
\item
\label{lkIuh727ii}
Let $x\in\intdom f$, let $y\in\intdom f$, let 
$(x_n)_{n\in\NN}\in(\dom f)^{\NN}$, and let
$(y_n)_{n\in\NN}\in(\intdom f)^{\NN}$. Suppose that 
$D^f(x_n,y_n)\to 0$. Then $x=y$
{\rm \cite[Theorem~3.9(iii)]{BB97}}.
\item
\label{lkIuh727iii}
Let $y\in\intdom f$. Then $D^f(\cdot,y)$ is coercive
{\rm \cite[Lemma~7.3(v)]{Ccm01}}.
\item
\label{lkIuh727iv}
Let $\{x,y\}\subset\intdom f$. Then 
$D^f(x,y)=D^{f^*}(\nabla f(y),\nabla f(x))$
{\rm \cite[Lemma~7.3(vii)]{Ccm01}}.
\end{enumerate}
\end{lemma}

\begin{proposition}
\label{4p:2}
In Problem~\ref{prob:1}, suppose that $\XX$ and $\YY$ 
are finite-dimensional. Let $h\in\Gamma_0(\XX)$ and 
$j\in\Gamma_0(\YY)$ be Legendre functions such that 
$\intdom f\subset\intdom h$, $L(\intdom f)\subset\intdom j$, 
and there exist $\varepsilon$ and $\delta$ in $\RPP$ such that 
$\nabla h+\varepsilon A$ and $\nabla j+\delta B$ are coercive. 
Let $\sigma\in\left[\max\{\varepsilon,\delta\},\pinf\right[$
and execute algorithm~\eqref{IVe:Bfg+uh4f9-09h}. 
Then $(x_n,y_n^*)\to(\overline{x},\overline{y}^*)$.
\end{proposition}
\begin{proof}
Set 
$\boldsymbol{C}=\boldsymbol{Z}\cap\overline{\dom}\boldsymbol{f}$.
We first observe that, as in \eqref{eKKnhf74g15a}, 
\begin{equation}
\label{4eKKnhf74g20T}
(x_n)_{n\in\NN}\in(\intdom f)^{\NN}\quad\text{and}\quad
(y^*_n)_{n\in\NN}\in(\intdom g^*)^{\NN}\quad\text{are bounded}.
\end{equation}
In addition, we deduce from \eqref{eKKnhf74g05},
\eqref{eKKnhf74g04i}, and Proposition~\ref{p:2}\ref{p:2i-} that 
$\overline{\boldsymbol{x}}=(\overline{x},\overline{y}^*)\in
\boldsymbol{C}\subset\bigcap_{n\in\NN}H^{\boldsymbol{f}}
(\boldsymbol{x}_0,\boldsymbol{x}_n)$, and hence from 
\eqref{eKKnhf74g19a} that
\begin{equation}
\label{4eKKnhf74g20u}
(\forall n\in\NN)\quad 
D^f(\overline{x},x_n)+D^{g^*}(\overline{y}^*,y_n^*)
=D^{\boldsymbol{f}}\big(\overline{\boldsymbol{x}},
\boldsymbol{x}_n\big)
\leq D^{\boldsymbol{f}}\big(\overline{\boldsymbol{x}},
\boldsymbol{x}_0\big).
\end{equation}
By virtue of Proposition~\ref{p:2}\ref{p:2iv} and 
\eqref{4eKKnhf74g20T}, it suffices to show that every cluster 
point of $(x_n,y_n^*)_{n\in\NN}$ belongs to $\boldsymbol{Z}$. 
To this end, take $x\in\XX$, $y^*\in\YY$, and 
a strictly increasing sequence $(k_n)_{n\in\NN}$ in $\NN$
such that $x_{k_n}\to x$ and $y_{k_n}^*\to y^*$. 
Then $Lx_{k_n}\to Lx$, $x\in\overline{\dom}f$, and 
$y^*\in\overline{\dom}g^*$. 
Since $\overline{x}\in\intdom f$ and since
\eqref{4eKKnhf74g20u} implies that 
$(D^f(\overline{x},x_{k_n}))_{n\in\NN}$ is bounded,
it follows from Lemma~\ref{lkIuh727}\ref{lkIuh727i} that
$x\in\intdom f$. Analogously, $y^*\in\intdom g^*$. 
In turn, Lemma~\ref{lkIuh727}\ref{lkIuh727i--}
asserts that 
\begin{equation}
\label{4ekIuh725c}
\nabla f\big(x_{k_n}\big)\to\nabla f(x)
\quad\text{and}\quad 
\nabla g^*\big(y_{k_n}^*\big)\to\nabla g^*(y^*).
\end{equation}
Furthermore, since $\inte\dom f\subset\inte\dom h$ 
and $L(\inte\dom f)\subset\inte\dom j$, we obtain
$x\in\intdom h$ and $Lx\in\intdom j$. Thus, there exists 
$\rho\in\RPP$ such that $B(x;\rho)\subset\intdom h$ and
$B(Lx;\rho)\subset\intdom j$.
We therefore assume without loss of generality that 
\begin{equation}
\label{4ekIuh716a}
(x_{k_n})_{n\in\NN}\in B(x;\rho)^{\NN}
\quad\text{and}\quad 
(Lx_{k_n})_{n\in\NN}\in B(Lx;\rho)^{\NN}.
\end{equation}
In view of Lemma~\ref{lkIuh727}\ref{lkIuh727i--},
$h({B(x;\rho)})$ and $\nabla h({B(x;\rho)})$ are therefore compact,
which implies that $(h(x_{k_n}))_{n\in\NN}$ and 
$(\nabla h(x_{k_n}))_{n\in\NN}$ are bounded. 
Hence $(D^h(\overline{x},x_{k_n}))_{n\in\NN}$ is bounded and, 
moreover, it follows from \eqref{IVe:Bfg+uh4f9-09h}, 
\eqref{4eKKnhf74g20T}, Lemma~\ref{IVlKKnhf74g12}, 
and \eqref{4:e12} that $(a_{k_n})_{n\in\NN}$ 
is a bounded sequence in $\intdom h$. 
We show likewise that $(D^j(L\overline{x},Lx_{k_n}))_{n\in\NN}$ 
and $(b_{k_n})_{n\in\NN}$ are bounded. 
Next, since the convexity of $h$ yields 
\begin{align}
(\forall n\in\NN)\;\;
D^h(\overline{x},a_{k_n})
&=h(\overline{x})-h(a_{k_n})-\pair{\overline{x}-a_{k_n}}
{\nabla h(a_{k_n})}\nonumber\\
&=h(\overline{x})-h(x_{k_n})-\pair{\overline{x}-x_{k_n}}
{\nabla h(x_{k_n})}+\Pair{\overline{x}-a_{k_n}}
{\nabla h(x_{k_n})-\nabla h(a_{k_n})}\nonumber\\
&\quad\;-\big(h(a_{k_n})-h(x_{k_n})-\pair{a_{k_n}-x_{k_n}}
{\nabla h(x_{k_n})}\big)\nonumber\\
&\leq D^h(\overline{x},x_{k_n})+\Pair{\overline{x}-a_{k_n}}
{\nabla h(x_{k_n})-\nabla h(a_{k_n})},
\end{align}
we derive from \eqref{IVe:Bfg+uh4f9-09h} that
\begin{align}
\label{4ekIuh722a}
(\forall n\in\NN)\;\;
\sigma^{-1}D^h(\overline{x},a_{k_n})
&\leq\gamma_{k_n}^{-1}D^h(\overline{x},a_{k_n})\nonumber\\
&\leq\varepsilon^{-1}D^h(\overline{x},x_{k_n})+
\Pair{\overline{x}-a_{k_n}}{\gamma_{k_n}^{-1}\big(
\nabla h(x_{k_n})-\nabla h(a_{k_n})\big)}\nonumber\\
&=\varepsilon^{-1}D^h(\overline{x},x_{k_n})+
\pair{\overline{x}-a_{k_n}}{a^*_{k_n}+L^*y^*_{k_n}}\nonumber\\
&=\varepsilon^{-1}D^h(\overline{x},x_{k_n})+
\pair{\overline{x}-a_{k_n}}{a^*_{k_n}+L^*\overline{y}^*}
\nonumber\\
&\quad\;+\pair{L\overline{x}-La_{k_n}}{y^*_{k_n}-\overline{y}^*}.
\end{align}
Similarly,
\begin{align}
\label{4ekIuh722b}
(\forall n\in\NN)\quad 
\sigma^{-1}D^j(L\overline{x},b_{k_n})
&\leq\delta^{-1}D^j(L\overline{x},Lx_{k_n})
+\pair{L\overline{x}-b_{k_n}}{b_{k_n}^*-y_{k_n}^*}\nonumber\\
&\leq\delta^{-1}D^j(L\overline{x},Lx_{k_n})
+\pair{L\overline{x}-b_{k_n}}{b_{k_n}^*-\overline{y}^*}
\nonumber\\
&\quad\;+\pair{L\overline{x}-b_{k_n}}{\overline{y}^*-y_{k_n}^*}.
\end{align}
Since \eqref{eKKnhf74g03} entails that
\begin{equation}
\label{ekIuh725a}
\overline{\boldsymbol{x}}\in\boldsymbol{C}\subset
\bigcap_{n\in\NN}\boldsymbol{H}_{k_n}=
\bigcap_{n\in\NN}\menge{(x,y^*)\in\XXX}
{\pair{x-a_{k_n}}{a_{k_n}^*+L^*y^*}+
\pair{Lx-b_{k_n}}{b_{k_n}^*-y^*}\leq 0},
\end{equation}
we deduce from \eqref{4ekIuh722a} and 
\eqref{4ekIuh722b} that
\begin{align}
&\hskip -6mm (\forall n\in\NN)\quad
\sigma^{-1}\big(D^h(\overline{x},a_{k_n})
+D^j(L\overline{x},b_{k_n})\big)\nonumber\\
&\leq\varepsilon^{-1}D^h(\overline{x},x_{k_n})
+\delta^{-1}D^j(L\overline{x},Lx_{k_n})
+\pair{L\overline{x}-La_{k_n}}{y^*_{k_n}-\overline{y}^*}
+\pair{L\overline{x}-b_{k_n}}{\overline{y}^*-y_{k_n}^*}
\nonumber\\
&=\varepsilon^{-1}D^h(\overline{x},x_{k_n})
+\delta^{-1}D^j(L\overline{x},Lx_{k_n})
+\pair{b_{k_n}-La_{k_n}}{y^*_{k_n}-\overline{y}^*}
\nonumber\\
&\leq\varepsilon^{-1}D^h(\overline{x},x_{k_n})
+\delta^{-1}D^j(L\overline{x},Lx_{k_n})
+(\|b_{k_n}\|+\|L\|\,\|a_{k_n}\|)(\|y^*_{k_n}\|+\|\overline{y}^*\|).
\end{align}
Hence, the boundedness of $(a_{k_n})_{n\in\NN}$, 
$(b_{k_n})_{n\in\NN}$, $(y^*_{k_n})_{n\in\NN}$, 
$(D^h(\overline{x},x_{k_n}))_{n\in\NN}$, 
and $(D^j(L\overline{x},Lx_{k_n}))_{n\in\NN}$ implies that of
$(D^h(\overline{x},a_{k_n}))_{n\in\NN}$ and 
$(D^j(L\overline{x},b_{k_n}))_{n\in\NN}$.
In turn, by Lemma~\ref{lkIuh727}\ref{lkIuh727iv},
\begin{equation}
\label{4ekIuh722c}
\Big(D^{h^*}\big(\nabla h(a_{k_n}),
\nabla h(\overline{x})\big)\Big)_{n\in\NN}
\quad\text{and}\quad
\Big(D^{j^*}\big(\nabla j(b_{k_n}),
\nabla j(L\overline{x})\big)\Big)_{n\in\NN}
\quad\text{are bounded}
\end{equation}
and, since Lemma~\ref{lkIuh727}\ref{lkIuh727iii} asserts 
that $D^{h^*}(\cdot,\nabla h(\overline{x}))$ 
and $D^{j^*}(\cdot,\nabla j(L\overline{x}))$ are coercive, 
it follows from \eqref{4ekIuh722c} that 
$(\nabla h(a_{k_n}))_{n\in\NN}$ 
and $(\nabla j(b_{k_n}))_{n\in\NN}$ are bounded. 
Thus, since $(y^*_{k_n})_{n\in\NN}$, $(\nabla h(x_{k_n}))_{n\in\NN}$ 
and $(\nabla j(Lx_{k_n}))_{n\in\NN}$ are bounded, 
we infer from \eqref{IVe:Bfg+uh4f9-09h} that 
\begin{equation}
\label{4ekIuh726a}
(a_{k_n}^*)_{n\in\NN}\quad\text{and}\quad (b_{k_n}^*)_{n\in\NN}
\quad\text{are bounded.}
\end{equation}
On the other hand, since \eqref{eKKnhf74g03} yields
$\overline{\boldsymbol{x}}\in\boldsymbol{C}
\subset\bigcap_{n\in\NN}H^{\boldsymbol{f}}(\boldsymbol{x}_{k_n},
\boldsymbol{x}_{k_n+1/2})$, \eqref{eKKnhf74g19a} and
\eqref{4eKKnhf74g20u} imply that 
\begin{equation}
\label{4ekIuh726d}
(\forall n\in\NN)\quad D^f(\overline{x},x_{k_n+1/2})
+D^{g^*}(\overline{y}^*,y_{k_n+1/2}^*)
=D^{\boldsymbol{f}}\big(\overline{\boldsymbol{x}},
\boldsymbol{x}_{k_n+1/2}\big)
\leq D^{\boldsymbol{f}}\big(\overline{\boldsymbol{x}},
\boldsymbol{x}_{k_n}\big)\leq 
D^{\boldsymbol{f}}\big(\overline{\boldsymbol{x}},
\boldsymbol{x}_0\big).
\end{equation}
Thus, Lemma~\ref{lkIuh727}\ref{lkIuh727iv} yields
\begin{align}
\label{4ekIuh727d}
& (\forall n\in\NN)\quad 
D^{f^*}\big(\nabla f(x_{k_n+1/2}),
\nabla f(\overline{x})\big)+D^g\big(\nabla g^*(y_{k_n+1/2}^*),
\nabla g^*(\overline{y}^*)\big)\nonumber\\
&\hskip 32mm=D^f(\overline{x},x_{k_n+1/2})
+D^{g^*}(\overline{y}^*,y_{k_n+1/2}^*)\nonumber\\
&\hskip 32mm\leq D^{\boldsymbol{f}}\big(\overline{\boldsymbol{x}},
\boldsymbol{x}_0\big)
\end{align}
and, since $D^{f^*}(\cdot,\nabla f(\overline{x}))$ and
$D^g(\cdot,\nabla g^*(\overline{y}^*))$ are coercive by
Lemma~\ref{lkIuh727}\ref{lkIuh727iii}, it follows that
\begin{equation}
\label{4ekIuh724}
(\nabla f(x_{k_n+1/2}))_{n\in\NN}\quad\text{and}\quad
(\nabla g^*(y_{k_n+1/2}^*))_{n\in\NN}
\quad\text{are bounded.}
\end{equation}
However, as in \eqref{IVe:IV4}, 
$D^f(x_{k_n+1/2},x_{k_n})\to 0$ and 
$D^{g^*}(y_{k_n+1/2}^*,y_{k_n}^*)\to 0$, and it therefore follows 
from Lemma~\ref{lkIuh727}\ref{lkIuh727iv} that
\begin{equation}
\label{4ekIuh725a}
D^{f^*}\big(\nabla f(x_{k_n}),\nabla f(x_{k_n+1/2})\big)\to 0
\quad\text{and}\quad 
D^g\big(\nabla g^*(y_{k_n}^*),\nabla g^*(y_{k_n+1/2}^*)\big)\to 0.
\end{equation} 
In view of Lemma~\ref{lkIuh727}\ref{lkIuh727ii}, we infer 
from \eqref{4ekIuh725c}, \eqref{4ekIuh724}, and 
\eqref{4ekIuh725a} that there exists a strictly increasing 
sequence $(p_{k_n})_{n\in\NN}$ in $\NN$ such that
\begin{equation}
\nabla f\big(x_{p_{k_n}+1/2}\big)\to\nabla f(x)
\quad\text{and}\quad 
\nabla g^*\big(y_{p_{k_n}+1/2}^*\big)\to\nabla g^*(y^*).
\end{equation}
Since, by Lemma~\ref{lkIuh727}\ref{lkIuh727i--}--%
\ref{lkIuh727i-},
$(\nabla f)^{-1}=\nabla f^*$ is continuous on
$\intdom f^*$ and $(\nabla g^*)^{-1}=\nabla g$ is continuous on
$\intdom g$, we obtain $x_{p_{k_n}+1/2}\to x$ and 
$y_{p_{k_n}+1/2}^*\to y^*$. Thus, 
\begin{equation}
\label{4ekIuh725b}
x_{p_{k_n}+1/2}-x_{p_{k_n}}\to 0
\quad\text{and}\quad 
y_{p_{k_n}+1/2}^*-y_{p_{k_n}}^*\to 0.
\end{equation}
On the other hand, as in \eqref{IVe:8re84},
\begin{multline}
(\forall n\in\NN)\quad
\|x_{p_{k_n}}-x_{p_{k_n}+1/2}\|\,
\|a_{p_{k_n}}^*+L^*b_{p_{k_n}}^*\|
+\|b_{p_{k_n}}-La_{p_{k_n}}\|\,
\|y_{p_{k_n}}^*-y_{p_{k_n}+1/2}^*\|\\
\geq\sigma^{-1}\big(D^h(x_{p_{k_n}},a_{p_{k_n}})
+D^j(Lx_{p_{k_n}},b_{p_{k_n}})\big),
\end{multline}
and hence, since $(a_{p_{k_n}})_{n\in\NN}$ and 
$(b_{p_{k_n}})_{n\in\NN}$ are bounded, 
we deduce from \eqref{4ekIuh726a} and 
\eqref{4ekIuh725b} that
\begin{equation}
\label{4eKKnhf74g2022a}
\begin{cases}
D^h\big(x_{p_{k_n}},a_{p_{k_n}}\big)\to 0\\
D^j\big(Lx_{p_{k_n}},b_{p_{k_n}}\big)\to 0\\
x_{p_{k_n}}\to x\\
Lx_{p_{k_n}}\to Lx\\
(a_{p_{k_n}})_{n\in\NN}\;\text{has a cluster point}\\
(b_{p_{k_n}})_{n\in\NN}\;\text{has a cluster point}.
\end{cases}
\end{equation}
Consequently, by dropping to a subsequence if necessary and invoking 
Lemma~\ref{lkIuh727}\ref{lkIuh727ii}, we get
\begin{equation}
a_{{p_{k_n}}}\to x\quad\text{and}\quad b_{{p_{k_n}}}\to Lx.
\end{equation}
Hence, using the fact that $\nabla h(x_{{p_{k_n}}})\to\nabla h(x)$ 
and $\nabla j(Lx_{{p_{k_n}}})\to\nabla j(Lx)$, we derive that
$\nabla h(x_{{p_{k_n}}})-\nabla h(a_{{p_{k_n}}})\to 0$ 
and $\nabla j(Lx_{{p_{k_n}}})-\nabla j(b_{{p_{k_n}}})\to 0$, 
which, in view of \eqref{IVe:Bfg+uh4f9-09h}, yields
\begin{equation}
\label{4ekIuh711f}
a^*_{{p_{k_n}}}+L^*y^*_{{p_{k_n}}}\to 0
\quad\text{and}\quad
b^*_{{p_{k_n}}}-y_{{p_{k_n}}}^*\to 0.
\end{equation}
Thus, since $y_{{p_{k_n}}}^*\to y^*$, it follows that 
$a_{{p_{k_n}}}^*\to -L^*y^*$ and $b_{{p_{k_n}}}^*\to y^*$. 
In summary,
\begin{equation}
\gra A\ni(a_{{p_{k_n}}},a_{{p_{k_n}}}^*)\to (x,-L^*y^*)
\quad\text{and}\quad
\gra B\ni(b_{{p_{k_n}}},b_{{p_{k_n}}}^*)\to (Lx,y^*).
\end{equation}
Since $\gra A$ and $\gra B$ are closed 
\cite[Proposition~20.33(iii)]{Livre1}, we conclude that
$(x,-L^*y^*)\in\gra A$ and $(Lx,y^*)\in\gra B$, and therefore
that $(x,y^*)\in\boldsymbol{Z}$.
\end{proof}

Let us note that, even in Euclidean spaces, it may be easier to 
evaluate $(\nabla h+\gamma\partial\varphi)^{-1}$ than the usual 
proximity operator
$\prox_{\gamma\varphi}=(\Id+\gamma\partial\varphi)^{-1}$ 
introduced by Moreau \cite{Mor62b}, which is
based on $h=\|\cdot\|^2/2$. We provide illustrations of such 
instances in the standard Euclidean space $\RR^m$.

\begin{example}
\label{IVex:1}
Let $\gamma\in\RPP$, let $\phi\in\Gamma_0(\RR)$ be such that 
$\dom\phi\cap\RPP\neq\emp$, and let $\vartheta$ be the 
Boltzmann-Shannon entropy function, i.e.,
\begin{equation}
\vartheta\colon\xi\mapsto
\begin{cases}
\xi\ln\xi-\xi,&\text{if}\;\;\xi\in\RPP;\\
0,&\text{if}\;\;\xi=0;\\
\pinf,&\text{otherwise}.
\end{cases}
\end{equation}
Set $\varphi\colon(\xi_i)_{1\leq i\leq m}\mapsto
\sum_{i=1}^m\phi(\xi_i)$ and $h\colon(\xi_i)_{1\leq i\leq m}
\mapsto\sum_{i=1}^m\vartheta(\xi_i)$. Note that $h$ is a 
supercoercive Legendre function \cite[Sections~5 and 6]{BB97} 
and hence Proposition~\ref{IVpp:20a}\ref{IVcl:20b4} asserts 
that $\nabla h+\gamma\partial\varphi$ is coercive and 
$\dom(\nabla h+\gamma\partial\varphi)^{-1}=\RR^m$. 
Now let $(\xi_i)_{1\leq i\leq m}\in\RR^m$, set 
$(\eta_i)_{1\leq i\leq m}
=(\nabla h+\gamma\partial\varphi)^{-1}(\xi_i)_{1\leq i\leq m}$, 
let $W$ be the Lambert function \cite{IVlamb96,Lamb58}, 
i.e., the inverse of $\xi\mapsto\xi e^{\xi}$ on $\RP$, 
and let $i\in\{1,\ldots,m\}$. Then $\eta_i$ can be computed 
as follows.
\begin{enumerate}
\item
\label{IVex:1ii}
Let $\omega\in\RR$ and suppose that 
\begin{equation}
\phi\colon\xi\mapsto
\begin{cases}
\xi\ln\xi-\omega\xi,&\text{if}\;\;\xi\in\RPP;\\
0,&\text{if}\;\;\xi=0;\\
+\infty,&\text{otherwise}.
\end{cases}
\end{equation}
Then $\eta_i=e^{(\xi_i+\gamma(\omega-1))/(\gamma+1)}$.
\item 
\label{IVex:1iii}
Let $p\in\left[1,+\infty\right[$ 
and suppose that either $\phi=|\cdot|^p/p$ or
\begin{equation}
\phi\colon\xi\mapsto
\begin{cases}
\xi^p/p,&\text{if}\;\;\xi\in\RP;\\
+\infty,&\text{otherwise}.
\end{cases}
\end{equation}
Then 
\begin{equation}
\eta_i=
\begin{cases}
\left(\dfrac{W(\gamma(p-1)e^{(p-1)\xi_i})}{\gamma(p-1)}
\right)^{\frac{1}{p-1}},&\text{if}\;\;p>1;\\[4mm]
e^{\xi_i-\gamma},&\text{if}\;\;p=1.
\end{cases}
\end{equation}
\item
\label{IVex:1v}
Let $p\in\left[1,+\infty\right[$ and suppose that
\begin{equation}
\phi\colon\xi\mapsto
\begin{cases}
\xi^{-p}/p,&\text{if}\;\;\xi\in\RPP;\\
+\infty,&\text{otherwise}.
\end{cases}
\end{equation}
Then 
\begin{equation}
\eta_i=\left(\frac{W(\gamma(p+1)e^{-(p+1)\xi_i})}{\gamma(p+1)}
\right)^{\frac{-1}{p+1}}.
\end{equation}
\item
\label{IVex:1vi}
Let $p\in\left]0,1\right[$ and suppose that
\begin{equation}
\phi\colon\xi\mapsto
\begin{cases}
-\xi^p/p,&\text{if}\;\;\xi\in\RP;\\
+\infty,&\text{otherwise}.
\end{cases}
\end{equation} 
Then 
\begin{equation}
\eta_i=\Bigg(\frac{W(\gamma(1-p)e^{(p-1)\xi_i})}{\gamma(1-p)}
\Bigg)^{\frac{1}{p-1}}.
\end{equation}
\end{enumerate}
\end{example}

\begin{example}
\label{IVex:2}
Let $\phi\in\Gamma_0(\RR)$ be such that 
$\dom\phi\cap\left]0,1\right[\neq\emp$ 
and let $\vartheta$ be the Fermi-Dirac entropy, i.e.,
\begin{equation}
\vartheta\colon\xi\mapsto
\begin{cases}
\xi\ln\xi+(1-\xi)\ln(1-\xi),&\text{if}\;\;\xi\in\left]0,1\right[;\\
0&\text{if}\;\;\xi\in\{0,1\};\\
\pinf,&\text{otherwise}.
\end{cases}
\end{equation}
Set $\varphi\colon(\xi_i)_{1\leq i\leq m}
\mapsto\sum_{i=1}^m\phi(\xi_i)$ 
and $h\colon(\xi_i)_{1\leq i\leq m}
\mapsto\sum_{i=1}^m\vartheta(\xi_i)$. 
Note that $h$ is a Legendre function \cite[Sections~5 and 6]{BB97} 
and that $\inte\dom h=\left]0,1\right[^m$ is bounded. Therefore,  
Proposition~\ref{IVpp:20a}\ref{IVcl:20b1} asserts that 
$\nabla h+\partial\varphi$ is coercive and that
$\dom(\nabla h+\partial\varphi)^{-1}=\RR^m$. 
Now let $(\xi_i)_{1\leq i\leq m}\in\RR^m$, 
set $(\eta_i)_{1\leq i\leq m}
=(\nabla h+\partial\varphi)^{-1}(\xi_i)_{1\leq i\leq m}$, 
and let $i\in\{1,\ldots,m\}$. Then $\eta_i$ can be computed 
as follows.
\begin{enumerate}
\item
\label{IVex:2i}
Let $\omega\in\RR$ and suppose that
\begin{equation}
\phi\colon\xi\mapsto
\begin{cases}
\xi\ln\xi-\omega\xi,&\text{if}\;\;\xi\in\RPP;\\
0,&\text{if}\;\;\xi=0;\\
\pinf,&\text{otherwise}.
\end{cases}
\end{equation}
Then $\eta_i=-e^{\xi_i+\omega-1}/2+\sqrt{e^{2(\xi_i+\omega-1)}/4
+e^{\xi_i+\omega-1}}$.
\item
\label{IVex:2ii}
Suppose that
\begin{equation}
\phi\colon\xi\mapsto
\begin{cases}
(1-\xi)\ln(1-\xi)+\xi,&\text{if}\;\;\xi\in\left]-\infty,1\right[;\\
1,&\text{if}\;\;\xi=1;\\
\pinf,&\text{otherwise}.
\end{cases}
\end{equation}
Then $\eta_i=1+e^{-\xi_i}/2-\sqrt{e^{-\xi_i}+e^{-2\xi_i}/4}$.
\end{enumerate}
\end{example}

\end{document}